\documentclass[12pt]{amsart}
\usepackage{amssymb,amsmath,amsfonts,amscd,euscript}
\usepackage{tikz}
\usepackage{tikz-cd}
\usepackage{fullpage}
\usepackage{stmaryrd}

\usepackage{hyperref}
\usepackage{todonotes}

\newcommand{\nc}{\newcommand}

\numberwithin{equation}{section}
\newtheorem{thm}{Theorem}[section]
\newtheorem*{thm*}{Theorem}
\newtheorem*{thma}{Theorem A}
\newtheorem*{thmb}{Theorem B}
\newtheorem*{thmc}{Theorem C}
\newtheorem{prop}[thm]{Proposition}
\newtheorem{lem}[thm]{Lemma}
\newtheorem{cor}[thm]{Corollary}
\theoremstyle{remark}
\newtheorem{rem}[thm]{Remark}
\newtheorem{definition}[thm]{Definition}
\newtheorem{example}[thm]{Example}
\newtheorem{dfn}[thm]{Definition}

\nc{\gl}{\mathfrak{gl}}
\nc{\GL}{\mathsf{GL}}
\nc{\g}{\mathfrak{g}}
\nc{\gh}{\widehat\g}
\nc{\h}{\mathfrak{h}}
\nc{\la}{\lambda}
\nc{\al}{\alpha }
\nc{\be}{\beta }
\nc{\ve}{\varepsilon }
\nc{\om}{\omega }

\nc{\ta}{\theta}
\nc{\ch}{{\mathop {\rm ch}}}
\nc{\Tr}{{\mathop {\rm Tr}\,}}
\nc{\Id}{{\mathop {\rm Id}}}
\nc{\ad}{{\mathop {\rm ad}}}
\nc{\bra}{\langle}
\nc{\ket}{\rangle}
\nc{\pa}{\partial}
\nc{\ld}{\ldots}
\nc{\cd}{\cdots}
\nc{\hk}{\hookrightarrow}
\nc{\T}{\otimes}
\nc{\gr}{\mathrm{gr}}
\nc{\ov}{\overline}

\nc{\cO}{\mathcal O}
\nc{\msl}{\mathfrak{sl}}
\nc{\mgl}{\mathfrak{gl}}
\nc{\U}{\mathrm U}
\nc{\V}{\EuScript V}
\nc{\cL}{\mathcal{L}}

\newcommand{\bC}{{\mathbb C}}
\newcommand{\bK}{{\mathbb C}}

\newcommand{\bZ}{{\mathbb Z}}

\newcommand{\bP}{{\mathbb P}}

\newcommand{\bW}{{\mathbb W}}

\newcommand{\fh}{{\mathfrak h}}
\newcommand{\fa}{{\mathfrak a}}
\newcommand{\fg}{{\mathfrak g}}

\newcommand{\fb}{{\mathfrak b}}

\newcommand{\fn}{{\mathfrak n}}
\newcommand{\A}{\EuScript{A}}

\newcommand{\bx}{{\bf x}}
\newcommand{\by}{{\bf y}}
\newcommand{\eO}{\EuScript{O}}

\newcommand{\Hom}{\mathrm{Hom}}

\nc{\cat}{\mathcal{C}}
\nc{\RepGL}[2]{{\mathsf{Rep}}(\GL_{#1})^{(#2)}}
\nc{\RepGLt}[2]{{\mathsf{Rep}}(\mgl_{#1}[t])^{(#2)}}
\nc{\RepSL}[2]{{\mathsf{Rep}}(\SL_{#1})^{(#2)}}
\nc{\RepSLt}[2]{{\mathsf{Rep}}(\msl_{#1}[t])^{(#2)}}
\nc{\Rep}{\mathsf{Rep}}
\nc{\de}{{\text -}}

\begin{document}

\title[Peter-Weyl, Howe and Schur-Weyl theorems for current groups]
{Peter-Weyl, Howe and Schur-Weyl theorems for current groups}

\author{Evgeny Feigin}
\address{Evgeny Feigin:\newline
Department of Mathematics,\newline
HSE University, Moscow, Usacheva str. 6, 119048, Russia,\newline
{\it and }\newline
Skolkovo Institute of Science and Technology, Skolkovo Innovation Center, Building 3,
Moscow 143026, Russia
}
\email{evgfeig@gmail.com}

\author{Anton Khoroshkin}
\address{Anton Khoroshkin:\newline
Department of Mathematics,\newline
HSE University, Moscow, Usacheva str. 6, 119048, Russia,\newline
{\it and }\newline
Institute for Theoretical and Experimental Physics, Moscow 117259, Russia;
}
\email{akhoroshkin@hse.ru}

\author{Ievgen Makedonskyi}
\address{Ievgen Makedonskyi:\newline
 JSPS International Research Fellow;
Department of Mathematics, Kyoto University, Oiwake,
Kita-Shirakawa, Sakyo Kyoto 606--8502, Japan}
\email{makedonskii\_e@mail.ru}

\begin{abstract}
The classical Peter-Weyl theorem describes the structure of the space of functions on a semi-simple
algebraic group. On the level of characters (in type A) this boils down to the Cauchy identity for
the products of Schur polynomials. We formulate and prove the analogue of the Peter-Weyl theorem
for the current groups. In particular, in type A the corresponding characters identity is governed by
the Cauchy identity for the products of q-Whittaker functions. We also formulate and prove a version of the
Schur-Weyl theorem for current groups. The link between the Peter-Weyl and Schur-Weyl theorems is provided
by the (current version of) Howe duality.
\end{abstract}

\maketitle

\section*{Introduction}
In this paper we formulate and prove current versions of three classical theorems from representation theory:
the (algebraic version of the) Peter-Weyl theorem, the Howe duality and the Schur-Weyl theorem. Let us recall the setup.

Let $G$ be an algebraic semi-simple simply-connected group over $\bC$ and let $\bC[G]$ be the space of algebraic functions on $G$.
Then $\bC[G]$ is naturally endowed with the commuting $G \times G$ action coming from the bimodule structure. The celebrated
Peter-Weyl theorem (see e.g. \cite{GW,TY}) states that $\bC[G]$ is isomorphic as $G\times G$ module to the direct sum
$\bigoplus_{\la\in P_+} V_\la\T V_\la^*$, where the sum is taken over the set $P_+$ of dominant integral weights $\lambda$ and
$V_\la$ is the irreducible $G$-module of highest weight $\la$.

The Howe duality \cite{Howe} describes the structure of the space of functions on the tensor product $V\T U$ with $\dim V=m$,
$\dim U=n$. More precisely, the Howe duality says that $S(V\T U)$ is isomorphic to the direct sum $V_\la\T U_\la$
of irreducible $\mgl_m$ and $\mgl_n$ modules ($\la$ runs over the set of partitions of length at most $\min(m,n)$).
Hence the Howe duality can be seen as an analogue of the Peter-Weyl theorem in the special type $A$ situation
(note however that $m, n$ might be different).

The celebrated Schur-Weyl duality states that the tensor power $V^{\T n}$ of the tautological (vector) representation of
$\mgl_m$ enjoys the decomposition into the direct sum of the tensor products $V_\lambda \otimes \mathbb{S}_\lambda$, where
$\mathbb{S}_\lambda$ is the Specht module over the symmetric group $\mathfrak{S}_n$ and $\la$ runs over the set of partitions
of length no greater than $\dim V$. It was shown by Howe that the Schur-Weyl theorem can be derived from the Howe duality.

In order to state our theorems let us introduce some notation.
Let $\eO=\bC[[t]]$ and let $G(\eO)$ be the corresponding current group over the  ring of formal series in one variable (see e.g.
\cite{Kum1}, section 13.2). Then the space of functions $\bC[G(\eO)]$ carries natural $G(\eO)\de G(\eO)$ bimodule structure.
It is natural to ask what is the structure of the space of functions on the current group as $G(\eO)\de G(\eO)$ bimodule.
In order to state the answer we recall the notion of the global and local Weyl modules.

Let $\fg$ be the Lie algebra of $G$ and let $\fg[t]=\fg\T\bK[t]$ be the corresponding current Lie algebra.
Let $\la$ be a dominant integral weight of $\fg$. Then the global Weyl module $\bW_\la$ (see \cite{CFK,FMO}) is a cyclic $\fg[t]$ module with
cyclic vector $w_\la$ such that $\fn_+[t] w_\la=0$ (where $\fn_+\subset\fg$ is the positive nilpotent subalgebra of $\fg$) and
$\U(\fg\T 1)w_\la\simeq V_\la$ (where $\U(\fg)$ denotes the universal enveloping algebra). The module $\bW_\la$ is
known to be infinite dimensional. However, its "size" is controlled by the local Weyl module $W_\la$, which is the quotient of
$\bW_\la$ by the relations $\fh\T t\bC[t] w_\la=0$, where $\fh\subset\fg$ is the Cartan subalgebra.  The modules $W_\la$ are finite
dimensional; in particular, if $\fg$ is of type $A$, then $W_\la$ is isomorphic as a $\fg\T 1$-module to a tensor product of
several fundamental $\fg$-modules. The connection between the local and global Weyl modules is governed by certain
commutative (infinite-dimensional) algebra $\A_\la$, which is isomorphic (as a vector space) to the weight $\la$ subspace of
$\bW_\la$ (see e.g. \cite{FL2}). More precisely, the algebra $\A_\la$ acts freely on $\bW_\la$, commutes with the action of the
current algebra $\fg[t]$ and the quotient is isomorphic to
$W_\la$; in particular, the ratio of the (graded) characters of the global and local Weyl modules is equal to
the graded character of $\A_\la$. It is worth mentioning that global/local Weyl modules as well as the algebras $\A_\la$ have a very clear categorical meaning in terms of a highest weight category (see recollections in Section~\S\ref{sec::HWC} and references therein).

Now we are in position to state our first main theorem.
\begin{thma}{\rm (Theorem~\ref{mainth} below)}
The dual space of functions on the group $G(\eO)$ admits a filtration such that the associated graded space is isomorphic
to the direct sum
\[
\bigoplus_{\la\in P_+}  \bW_\la\T_{\A_\la} \bW_{\la}^o.
\]
\end{thma}
Here $\bW_{\la}^o$ is the right Weyl module (see Chapter \ref{WeylModules}).
It is isomorphic to the left Weyl module $\bW_{\la^*}$ as the Hopf algebra module, where
$\la^*=-w_0\la$. So
there is the natural vector space isomorphism:
\[
\bigoplus_{\la\in P_+}  \bW_\la\T_{\A_\la} \bW_{\la}^o\simeq \bigoplus_{\la\in P_+}  \bW_\la\T_{\A_\la} \bW_{\la^*}.
\]
We note that in contrast with the classical situation the matrix coefficients of the global Weyl modules do
not span the whole space of functions on the group $G(\eO)$.

% and $A_\la$ acts on $\bW_{\la}^o$ via the identification of
%$A_\la$ with the lowest weight space of $\bW_\la^*$.

%Recall that the Peter-Weyl theorem has an analogue for the algebra of functions on the space of matrices -- the Howe duality \cite{Howe}.
%Namely, the group $GL_n\times GL_m$ naturally acts  on the space of $n\times m$ matrices ${\rm Mat}_{n,m}$.
%Then the space of functions $\bC[{\rm Mat}_{n,m}]$ is isomorphic
%as $GL_n\times GL_m$-module to the direct sum $\bigoplus V_\la\T V_\la^*$, where the sum is taken over all weights $\la$
%of the form $\la=(\la_1\ge\dots\ge \la_n\ge 0)$ (we assume $n\le m$).

Now let us describe the current version of the Howe duality.
Recall that the character of $V_\la$ is given by the
Schur function $s_\la$. Hence the Howe duality on the level of characters boils down to the celebrated Cauchy identity:
\[
\prod_{i=1}^n\prod_{j=1}^m (1-x_iy_j)^{-1}=\sum_{\la=(\la_1\ge\dots\ge\la_n\ge 0)} s_\la(x_1,\dots,x_n)s_\la(y_1,\dots,y_m).
\]
Recall that the $q$-Whittaker functions $p_{\la}(\bx,q)$, $\bx=(x_1,\dots,x_n)$ are certain polynomials in $x_i$ and $q$
(see \cite{BF1,E,I,GLO1,GLO2})).
These polynomials are $t=0$ specializations of the Macdonald polynomials. The $q$-Whittaker functions
enjoy many nice
properties; in particular, they satisfy the following generalized Cauchy identity (which is the special case of the
similar identity for Macdonald polynomials, see \cite{M}):
\[
\prod_{i=1}^n\prod_{j=1}^m \prod_{k\ge 0}(1-x_iy_jq^k)^{-1}=
\sum_{\la=(\la_1\ge\dots\ge\la_n\ge 0)} p_\la(\bx,q)p_\la(\by,q)\prod_{a=1}^n (q)^{-1}_{\la_a-\la_{a+1}},
\]
where $(q)_r=\prod_{l=1}^r (1-q^l)$ and we assume $n\le m$,  $\la_{n+1}=0$ (see \cite{BP,BC} for the description
and properties of the corresponding measure
on the set of partitions, which generalizes the Schur measure of \cite{Ok,OR}). The left hand side of the Cauchy identity for the
$q$-Whittaker functions can be interpreted as the character of the space of algebraic functions on the space $V\T U\T\bC[[t]]$.
We prove the following theorem.
\begin{thmb}{\rm (Theorem~\ref{HoweCurrent} below)}
There exists a filtration on the dual space of functions on the space $V\T U\T\bC[[t]]$ such that the associated graded space is isomorphic
as $\mgl_V[t]\de \mgl_U[t]$ bimodule to the direct sum
\[
\bigoplus_{\la=(\la_1\ge\dots\ge \la_n\ge 0)}  \bW_\la(V)\T_{\A_\la} \bW_{\la}^o(U).
\]
\end{thmb}
Here $\bW_\la$ is a certain $\mgl$-analogue of the global type $A$ Weyl module and
$\A_\la$ is an analogue of the highest weight algebra. We note that the $q$-Whittaker functions are equal to the
characters of the local Weyl modules in type $A$ and the character of the algebra $\A_\la$ is given by
the product $\prod_{a=1}^n (q)^{-1}_{\la_a-\la_{a+1}}$.

Classical Schur-Weyl duality gives an equivalence of categories of modules over the Lie algebra $\msl_m$ of weight $\leq n \omega_1$
and the category of representations of symmetric group $\mathfrak{S}_n$ whenever $m>n$. This equivalence is given by the bimodule
$V_{\omega_1}^{\T n}\simeq \bigoplus_{\lambda \vdash n}V_{\lambda} \otimes\mathbb{S}_{\lambda}$.
Note that $V_{\omega_1}$ is the tautological representation $V$ of the Lie algebra $\msl_V$.

 We prove the current analogue of this theorem, where the Lie algebra $\msl_m$ is replaced by the current Lie algebra $\msl_m[t]$, the tautological representation $V_{\omega_1}\simeq V$ is replaced by the global Weyl module $\bW_{\omega_1}\simeq V[t]$ and the symmetric group $\mathfrak{S}_n$ is replaced by  the algebra $\bC[\mathfrak{S}_n]\ltimes \bC[t_1,\ldots,t_n]$. The
 category of its representations is no more a semisimple category but admits a structure of a highest weight category (see Section~\S\ref{sec::HWC} for definition). Let $\mathbb{K}_{\lambda}$ be the standard modules in this category
(called global Kato modules). In particular, the automorphism algebra of $\mathbb{K}_{\lambda}$ is equal to $\A_{\lambda}$.

 \begin{thmc}{\rm (Theorems~\ref{thm::SW::currents} and~\ref{Schur-Weylcurrent} below)}	
 The bimodule $(\mathbb{W}_{\omega_1})^{\otimes n}$ gives an equivalence of the Serre subcategory $\Rep(\msl_m[t])^{\leq n\omega_1}$
 and the category $\Rep(\mathfrak{S}_n\ltimes \bK[t_1,\ldots,t_n])$ whenever $m>n$. Moreover $(\mathbb{W}_{\omega_1})^{\otimes n}$ admits a filtration
 with subquotients isomorphic to $\mathbb{W}_{\lambda}\otimes_{\A_{\lambda}}\mathbb{K}_{\lambda}$.
  \end{thmc}

Finally, Theorem~{\bf{C}} can be used to deduce a numerical information on local/global Kato modules. In particular, we are able to index a basis of the local Weyl module $K_{\lambda}$ by fillings of Young diagrams in Section~\S\ref{sec::Kato::fillings} and find   a filtration by global Weyl modules of the wedge powers $\Lambda^n \bW_{\omega_1}$ (see Theorem~\ref{thm::Lambda::SW} below).

Our paper is organized as follows. In Section~\S\ref{Notation} we introduce the main algebraic and geometric objects we use in the paper.
We also recall basic facts and constructions to be used in the next sections.
In Section~\S\ref{PWC} we formulate and prove the Peter-Weyl theorem for current groups (Theorem {\bf{A}}).
In Section~\S\ref{TypeA} we deal with the type A case (mostly with $\fg=\mgl_n$)  and prove Theorem {\bf{B}}.
In Section~\S\ref{SW} we prove Theorem {\bf{C}}, deduce some properties of Kato modules from \cite{Kat2} and finish with the description of the wedge power $\Lambda^n \bW_{\omega_1}$ and related identities for characters.

\section{Notation and background}\label{Notation}
\subsection{Notation}
Let $\fg$ be a simple finite-dimensional Lie algebra over $\bC$ .
\begin{rem}
Throughout the paper we work over the complex numbers. However, we expect that most of our results hold over
an arbitrary algebraically closed field of characteristic zero.
\end{rem}

\begin{rem}
We note that many of the results and constructions below are valid for arbitrary reductive Lie algebras.
However, when passing to the theory of Weyl modules over the current algebra, the (semi)simplicity
becomes crucial. In particular, the $\mgl_n$-case requires separate definitions.
\end{rem}

Let $G$ be the corresponding simple simply-connected algebraic group.
Let $\Delta=\Delta_+\cup \Delta_-$ be the root system, $\fg=\fn_-\oplus\fh\oplus \fn_+$ be the triangular decomposition,
$W$ be the Weyl group. We denote by $r=\dim\fh$ the rank of $\fg$ and let $\langle\cdot,\cdot,\rangle$ be the Killing
form on $\fh^*$. Let $e_{\alpha}$ be the root vector for
$\alpha \in \Delta_+$, $f_{\alpha}$ be the root vector for $\alpha \in \Delta_-$.
We denote by $\{\alpha_i\}_{i=1}^r$, $\{\omega_i\}_{i=1}^r$ the sets of simple roots and fundamental weights.
Let $P$ be the weight lattice, $P_+$ the semigroup of dominant integral weights, i.e.
$P=\bigoplus_{i=1}^r \bZ\om_i$, $P_+=\bigoplus_{i=1}^r \bZ_{\ge 0}\om_i$. For $\lambda \in P_+$ we denote by
$V_{\lambda}$ the (left) irreducible $\fg$-module with the highest weight $\lambda$. This module is cyclic
with the cyclic vector $v_{\lambda}$ and the following relations:
\begin{equation}
e_{\alpha}v_{\lambda}=0, \alpha \in \Delta_+,\ (f_{-\alpha_i})^{\langle \alpha_i^\vee,\lambda \rangle+1}v_{\lambda}=0,\
hv_{\lambda}=\lambda(h) v_{\lambda}, h\in\fh
\end{equation}
(here $\alpha_i^\vee$ are simple coroots; in particular, if $\la=\sum_{i=1}^r m_i\om_i$, then $\langle \alpha_i^\vee,\lambda \rangle=m_i$).

Similarly one defines the right module $V^o_{\lambda}$ over $\fg$ (i.e. for any $x,y\in\fg$ the commutator
$[x,y]$ acts as $yx-xy$) with the highest weight $\lambda$:
\begin{equation}
v_{\lambda}^o f_{-\alpha}=0, \alpha \in \Delta_+,\ v_{\lambda}^o (e_{\alpha_i})^{\langle \alpha_i^\vee,\lambda \rangle+1}=0,\
v_{\lambda}^o h=\lambda(h)v_{\lambda}, h\in\fh.
\end{equation}
\begin{rem}\label{opp1}
We note that a (left) module $\pi:\fg\to\mgl(V)$ induces the right module ${\rm opp}\circ\pi:\fg\to \mgl(V)$
by the formula $x({\rm opp}\circ\pi)=-\pi(x)$.
The right module thus obtained will be denoted by ${\rm opp} (V)$ (in particular,
${\rm opp} (V)$ coincides with $V$ as a vector space). Then one easily sees that
$$V_\la^o\simeq {\rm opp} (V_\la^*),$$ i.e.
to obtain the right module $V_\la^o$ one has to take the dual (left) module $V_\la^*$ and to negate all the Lie algebra operators.
\end{rem}

\begin{rem}\label{opp2}
Below we will also consider the right representations of the Lie groups (i.e. the map $\pi:G\to {\rm GL}(V)$
such that $\pi(g_1g_2)=\pi(g_2)\pi(g_1)$). Clearly, given a (left) $G$-module $V$ one constructs the right
$G$ module ${\rm opp}(V)$ by inverting all the  operators corresponding to the Lie group elements. Obviously, the procedure sends
$V_\la^*$ to $V_\la^o$.
\end{rem}

\subsection{Peter-Weyl theorem}
The group $G\times G$ acts on $G$ by the multiplication from the left and from the right. More precisely, we set
$(g_1,g_2).g=g_1gg_2$. So one copy of $G$ acts by the (usual) left action, and the other acts via the right action.
Let $\bC[G]$ be the space of algebraic functions on $G$. The natural $G\de G$-bimodule structure of $\bC[G]$ is explicitly
given by $((g_1,g_2).\Psi)(g)=\Psi(g_1^{-1}gg_2^{-1})$. The celebrated Peter-Weyl theorem (see e.g. \cite{GW}, Theorem 4.2.7.)
describes the structure of this bimodule .
\begin{thm}\label{PeterWeylThm}
\[\bC[G]\simeq\bigoplus_{\lambda\in P_+}V_{\lambda}\otimes V^o_{\lambda}.\]
\end{thm}

\begin{rem}
The Peter-Weyl theorem is sometimes stated in the following form: $\bC[G]\simeq \bigoplus_{\la\in P_+} V_\la\T V_\la^*$.
This form of the theorem is valid if one considers the genuine (left-left) action of the group $G\times G$ on the algebra of functions
$\bC[G]$:  $((g_1,g_2).\Psi)(g)=\Psi(g_1^{-1}gg_2)$. The two formulations are related by the simple observation from Remarks \ref{opp1},
\ref{opp2}. The reason we prefer the left-right action is that in order to formulate the Peter-Weyl theorem for current algebras
we need to tensor two infinite-dimensional modules over the current algebra. In such a situation we prefer not to work with dual spaces.
\end{rem}

\subsection{Howe duality}
Howe duality \cite{Howe} can be regarded as a version of the Peter-Weyl theorem for the space of matrices ${\rm Mat}_{m,n}$.
Let us recall some details.
The weight lattice of $\mathfrak{gl}_n$
is equal to $P=\mathbb{Z}^n=\bigoplus_{i=1}^n \bZ \epsilon_i$. Let
\[
P_+=\lbrace \la_1\epsilon_1+ \dots +\la_n\epsilon_n |\la_1\geq \la_2 \geq \dots \geq \la_n \geq 0\rbrace\subset P
\]
be the set of dominant integral weight. We denote by $V_{\lambda}$ the irreducible highest weight $\gl_n$ module with
highest weight $\lambda$.

Denote the weight $\sum_i \la_i\epsilon_i$ by $\lambda=(\la_1, \dots, \la_n)$.
Note that the restriction of the $\mgl_n$ module $V_{\lambda}$ to $\msl_n$ is isomorphic to the irreducible representation
$V_{\sum_{i=1}^{n-1} (\la_i-\la_{i+1})\om_i}.$ Each dominant integral weight of $\gl_n$ corresponds to the Young
diagram with at most $n$ rows in the following way. The diagram $Y_\la$ attached to $\la$ contains $\la_1$ boxes in the first row,
$\la_2$ boxes in the second row, etc.
For example, for  $\lambda=(4,3,1)$, the diagram $Y_\la$ is of the following form:
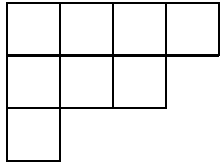
\[
\begin{picture}(140.00,70.00)(-5,00)

 \put(0.00,10.00){\line(1,0){20}}
 \put(0.00,30.00){\line(1,0){60}}
 \put(0.00,50.00){\line(1,0){80}}
 \put(0.00,70.00){\line(1,0){80}}

 \put(0.00,10.00){\line(0,1){60}}
 \put(20.00,10.00){\line(0,1){60}}
 \put(40.00,30.00){\line(0,1){40}}
 \put(60.00,30.00){\line(0,1){40}}
 \put(80.00,50.00){\line(0,1){20}}

\end{picture}\]
For a weight $\la$ we denote by $|\lambda|$ the number of boxes in the Young diagram $Y_\la$.
Now let $V$ be an $m$-dimensional vector space and let $U$ be an $n$-dimensional vector space.
Then $V \otimes U^o$ has a natural structure
of $\gl_m\de \gl_n$-bimodule (i.e. $\mgl_m$ acts from the left and $\mgl_n$ acts from the right).
We recall that $U^o$ is isomorphic to the dual space $U^*$ as a vector space. Thus the tensor product
$V\T U^o$ is naturally identified with the space of homomorphisms ${\rm Hom}(U,V)\simeq {\rm Mat}_{m,n}$.

We consider the $N$-th symmetric power of the $\gl_m\de \gl_n$-bimodule $S^N(V \otimes U^o)$.
\begin{thm}\label{Howeduality}
One has the isomorphism of $\mgl_n\de\mgl_m$ bimodules:
\[
S^N(V \T U^o)\simeq \bigoplus_{\la\vdash N} V_\lambda \otimes U^o_{\lambda},
\]
where $\la_{n+1}=\la_{m+1}=0$ in the right hand side.
\end{thm}

\subsection{Highest weight categories}
\label{sec::HWC}
In this section we recall the notion of a \emph{highest weight category} introduced in~\cite{CPS}.

Let $\cat$ be a $\bK$-linear abelian category whose simple objects are indexed by elements of a given set $\Upsilon$. Let $\pi:\Upsilon\twoheadrightarrow \Lambda$ be a surjection on a partially ordered set $(\Lambda,\leq)$.
\footnote{In \cite{CPS} one deals with finite posets $(\Lambda,\leq)$. However, as it is shown in~\cite{Kh} it is enough to have
slightly more general finiteness conditions. For example, one can deal with the poset $(\Lambda,\leq)$ whose
subposet $(\Lambda^{\leq\lambda},\leq)$ is finite for all $\lambda\in\Lambda$.}
Denote by $\cat^{\leq \lambda}$ the Serre subcategory generated by simples $L(\mu)$ with $\pi(\mu)\leq \lambda$.
Respectively, by $\cat^{=\lambda}$ we denote the quotient category $\cat^{\leq \lambda}/\cat^{<\lambda}$. Thus, for each $\lambda$ we have a pair of functors between corresponding categories.
\begin{dfn}
	A category $\cat$ is called a \emph{Highest Weight Category} iff $\forall \lambda\in \Lambda$ the embedding $\cat^{\leq\lambda}\hookrightarrow \cat$ is fully faithful on the level on derived categories:
	\[\forall M,N\in\cat^{\leq\lambda} \text{ there is an isomorphism } {\mathrm{RHom}}_{\cat^{\leq\lambda}}(M,N)  \simeq  {\mathrm{RHom}}_{\cat}(M,N). \]
\end{dfn}
We suppose that the category $\cat$ contains enough projectives (injectives) and we denote by $P(\lambda)$ a projective cover of $L(\lambda)$\footnote{All projective covers of an irreducible object are isomorphic but not up to a canonical isomorphism.}.
Denote by $\Delta(\lambda)$ a projective cover of $L(\lambda)$ in the subcategory $\cat^{\leq\lambda}$ which is called a \emph{standard object}. Respectively, we denote by $\nabla_{\lambda}$ (called costandard object) an injective hull of $L(\lambda)$ in $\cat^{\leq\lambda}$. By $\A_{\lambda}$ we denote the algebra of endomorphisms of $\Delta(\lambda)$.
We end up with the following diagram of (adjoint) functors of abelian categories:
\[
\begin{tikzcd}
\cat
\arrow[rrr,shift left=1ex,"\imath_{\lambda}^{!}(M):=M/((M)^{\not\leq\lambda})"]
&&&
\cat^{\leq\lambda}
\arrow[lll,shift left=1ex,"\imath_{\lambda}","\perp"']
\arrow[rrr,dotted,"{r_{\lambda}}" description,"\perp" near start,"\perp"' near start]
&&&
\cat^{=\lambda} \simeq \A_{\lambda}\text{-}mod
\arrow[lll,shift right=2ex,"r_\lambda^!(\text{-}):=\Delta(\lambda)\otimes_{\A_\lambda}\text{-}"']
\arrow[lll,shift left=2ex,"r_{\lambda}^{*}(\text{-}):= Hom_{\A_\lambda}\left(\Delta(\lambda) \text{,-}\right)"]
%{\A_\lambda}(\Delta(\lambda),\text{-})"]
\end{tikzcd}
\]
\emph{A proper standard} object $\bar{\Delta}(\lambda)$ is defined to be the image of $L(\lambda)$ after applying the left
adjoint functor $r_{\lambda}^{!}$. (Respectively, a proper costandard $\bar\nabla(\lambda)$ is defined to be the image of the
right adjoint $r_\lambda^{*}(L(\lambda))$).
Finally we are able to recall several equivalent definitions of a highest weight category:
\begin{thm}(\cite{CPS})
The following is equivalent for a $\bK$-linear abelian category $\cat$ whenever the (finite) set of irreducible objects $\Upsilon$ is partially ordered:
\begin{itemize}
	\item[(S1)] $\cat$ is a highest weight category;
	\item[(S2)] for all $\lambda\in \Upsilon$ a projective cover $P(\lambda)$ admits a filtration whose successive subquotients are isomorphic to standard modules $\Delta(\mu)$ with $\mu\ge\lambda$;
	\item[(S3)] For all $\lambda,\mu\in \Upsilon$ we have the following vanishing conditions
	%(s4)
	\begin{equation}
	\dim Ext_{\cat}^{i}(\Delta(\lambda),\bar{\nabla}(\mu)) = \left\{
	\begin{array}{l}
	1, \text{ if } \lambda=\mu \ \& \ i=0, \\
	0, \text{ otherwise. }
	\end{array}
	\right.
	\end{equation}
	\item[(S4)] The second extension group $Ext_{\cat}^{2}(\Delta(\lambda),\bar{\nabla}(\mu))=0$ vanishes for all $\lambda,\mu\in\Upsilon$.
	%(s3)
\end{itemize}
\end{thm}
The notion of a highest weight category $\cat$ was initiated by the following equality of multiplicities:
\begin{cor}
\label{cor::BGG}	
	The BGG (Bernstein-Gelfand-Gelfand) reciprocity holds in a highest weight category:
\[ \forall \lambda\leq\mu \text{ one has equality of multiplicities }[P(\lambda): \Delta(\mu)] = [\bar{\nabla}(\mu):L(\lambda)].\]
\end{cor}

\begin{cor}
\label{cor::HW::leq}	
	A subcategory $\cat^{\leq \lambda}$ of a highest weight category $\cat$ is a highest weight category.
\end{cor}

All aforementioned definitions were recollected in order to be able to formulate the following well known result~\cite{BBCKL,CI,Kh}.
\begin{example}
\label{ex::SL::HW}	
The category of finitely-generated graded $\fg[t]$-modules with finite-dimensional graded components is a highest weight
category with respect to the standard partial ordering of dominant weights $P_+$ that index the set of irreducible objects
	in this category:
	\begin{equation}
	\label{eq::def::ordering}
	\lambda \geq_{\fg} \mu \stackrel{def}{\Longleftrightarrow} \lambda -\mu \in P_+.
	\end{equation}
	\\
	A representation-theoretic description of standard  and proper standard  modules in these categories is given in the next subsection,
	where we use more common names and notation: \emph{global}  \emph{Weyl modules} $\bW_{\lambda}$ stand for standard modules
	$\Delta(\lambda)$ and \emph{local} \emph{Weyl modules} $W_{\lambda}$ are used instead of proper standard modules $\bar{\Delta}(\lambda)$.
Note also that an irreducible $\fg$ module $V_\la$ can show up as a subquotient of a graded $\fg[t]$ module in different homogeneous
components. Hence, strictly speaking, the set of irreducibles $\Upsilon$ in our category consists of pairs
$(\la,k)$, $\la\in P_+$, $k\in \bZ$.
One has an obvious forgetful map  $\pi:\Upsilon=P_+\times\mathbb{Z}\twoheadrightarrow P_+$ to the partially ordered set $(P_+,\leq_{\fg})$ that defines a structure of the Highest Weight Category.
In particular, in order to control the characters of the graded $\fg[t]$ modules (in particular, when using the BGG reciprocity) one has
to keep track of the graded component where an irreducible module $V_\la$ shows up.
\end{example}

\subsection{Weyl modules}\label{WeylModules}
In this section we consider representations of the current Lie algebra $\fg[t]=\fg\T\bC[t]$,
where $\fg$ is a simple finite-dimensional Lie algebra (we will also consider separately the case of $\fg=\mgl_n$).
For $x\in\fg$, $k\ge 0$ we denote by $xt^k$ the element $x\T t^k\in\fg[t]$.
For a dominant integral weight $\la$ we denote by $W_\la$ the corresponding
local Weyl module of highest weight $\la$ and by $\mathbb{W}_\la$ the global Weyl module (see e.g. \cite{CFK}).
The global Weyl module $\mathbb{W}_\la$ is cyclic $\fg[t]$ module with cyclic vector $w_\la$ of $\fh\T 1$ weight $\la$
and defining relations
\[
\fn_+\T\bC[t]. w_\la=0,\ (f_{-\al}\T 1)^{\langle \la,\al^\vee\rangle+1}w_\la=0.
\]
The defining relations for the local Weyl module $W_\la$ differ by the additional relation $\fh\T t\bC[t]. w_\la=0$.
Thus there is a natural surjective homomorphism of $\fg[t]$-modules $\bW_\la\to W_\la$.
We note that both local and global Weyl modules are graded by the action of the Cartan subalgebra.
Apart from that, Weyl modules enjoy an additional $q$-grading, defined by
${\rm deg}_q w_\la=0$ and ${\rm deg}_q xt^i=i$ (i.e. $xt^i$ increases the degree by $i$). We note that the subspace
of vectors of a fixed $q$-degree in a Weyl module is naturally a $\fg$-module.
For a graded $\fg$-module $M=\bigoplus_{k\in\mathbb{Z}} M[k]$ with
finite-dimensional graded components we denote $\ch_q M=\sum_{k\in \mathbb{Z}} q^k \ch M[k]$, where $\ch M[k]$ is the standard
character of the $\fg$-module $M[k]$.

The right global Weyl module $\bW_\la^o$ is defined in a similar fashion: this is a cyclic right $\fg[t]$-module with a cyclic vector
$w^o_\la$ of $\fh\T 1$ weight $\la$ subject to the relations:
\[
w^o_\la.\fn_-\T\bC[t]=0,\ w^o_\la.(e_{\al}\T 1)^{\langle \la,\al^\vee\rangle+1}=0.
\]
Similarly one defines local right Weyl modules just adding additional relation $w_\la^o.\fh\T t\bC[t]=0$.
One has the following lemma (compare with Remarks \ref{opp1} and \ref{opp2}).
\begin{lem}\label{Wopp}
One has the isomorphisms of vector spaces $\bW_\la^o\simeq \bW_{\la^*}$, $W_\la^o\simeq W_{\la^*}$. The right module structure
on the Weyl modules $\bW_{\la^*}$ and $W_{\la^*}$ is obtained by  negating all the Lie algebra operators.
\end{lem}
\begin{proof}
Let $w_0$ be the longest element in the Weyl group; recall $\la^*=-w_0\la$. Hence $\bW_{\la^*}$ contains a vector $u$ of weight
$-\la$ (the space of such vectors is one dimensional and of $q$-degree zero).
Now the Weyl module $\bW_{\la^*}$ is generated
from $u$ by the action of the universal enveloping algebra $\U(\fb_+)$. We also note that the map
\[
y_1\dots y_s u\mapsto (-1)^s w_\la^o y_s\dots y_1
\]
produces the isomorphism of vector spaces $\bW_{\la^*}\simeq \bW_\la^o$ with the desired link between the left-right actions.
The same argument works for the local Weyl modules and their opposite analogues.
\end{proof}

\begin{rem}
We note that $\bW_{\la^*}$ is very different from the restricted dual module $\bW^*_{\la}$ (for example, the former is cyclic module and the
latter is cocyclic).
\end{rem}

Recall the standard notation $(q)_n=\prod_{i=1}^n (1-q^i)$.
Let $\A=\U(\fh\T t\bC[t])$ be the universal enveloping algebra of the Lie algebra of currents over Cartan subalgebra with
trivial constant term. One has $\A=\bC[h_{\alpha_i}t^k]$, $1\le i\le r$, $k>0$.
Let $a$ be an element of $\A$. Then the submodule
$\U(\fg[t])aw_{\lambda}\subset \bW_\la$ is a cyclic module which satisfies all defining relations of $\bW_\lambda$.
Therefore $\A$ acts as algebra of endomorphisms of global Weyl module. More precisely, consider the map
\[
\varphi: \A\to \bW_\la,\ \varphi(a)=aw_{\lambda}.
\]
Let $\A_\la=\A/\ker \varphi$ be the quotient algebra. In particular, $\A_\la$ is isomorphic to the space of the $\fh$-weight $\la$
vectors in $\bW_\la$. Then the algebra $\A_\la$ acts on the global Weyl module $\bW_\la$ in such a way that this action commutes with the
$\fg[t]$-action.

We have the following properties of Weyl modules (see \cite{CFK,FL2,N}).
\begin{itemize}
\item $W_\la$ is finite-dimensional and $\dim W_\la=\prod_{i=1}^r (\dim W_{\om_i})^{\langle \la,\al_i^\vee\rangle}$;
\item  $\U(\fg)w_\la\simeq V_\la\subset W_\la$;
\item  the action of the algebra $\A_\la$ on $\bW_\la$ is free;
\item if $\la=\sum_{i=1}^r m_i\omega_i$, then $\A_\la$ is isomorphic to the tensor product
$\bigotimes_{i=1}^r \bC[x_{i,1},\dots,x_{i,m_i}]^{\mathfrak{S}_{m_i}}$ of
$r$ algebras of symmetric polynomials in $m_i$ variables; in particular, $h_{\alpha_i}t^k$ corresponds to $x_{i,1}^k+\dots+x_{i,m_i}^k$;
\item $\ch_q \mathbb{W}_\la=\ch_q W_\la\prod_{i=1}^r (q)^{-1}_{(\la,\al_i^\vee)}$;
\item the quotient of $\bW_\la$ by the action of $\A_\la^+$ (the augmentation ideal of $\A_\la$) is isomorphic to $W_\la$.
\end{itemize}

Finally, we denote by $P_\la$ the $\fg[t]$-module which is the projective cover of the irreducible module $V_\la$ (see e.g.
\cite{BBCKL,CI,Kh}).
Explicitly, $P_\la={\rm Ind}_{\fg}^{\fg[t]} V_\la$.
\begin{rem}
We have the isomorphism of vector space $P_\la\simeq V_\la\T (\fg\T t\bC[t])$.
\end{rem}

\subsection{The $\mgl_n$ Weyl modules}
In order to formulate the analogue of the Howe duality for current groups we will need the notion of
the global Weyl module over $\gl_n$. Let $\lambda=\sum_{i=1}^n \la_i \epsilon_i$ be a dominant integral weight for $\mgl_n$.
We define the global Weyl module as follows. Let $\bar \la=\sum_{i=1}^{n-1} (\la_i-\la_{i+1})\om_i$ be the corresponding
$\msl_n$ weight and consider the global Weyl module $\mathbb{W}_{\bar\lambda}$ for $\msl_n$. Let ${\rm Id}\in\mgl_n$ be the identity
matrix. In particular, $\mgl_n[t]=\msl_n[t]\oplus {\rm Id}[t]$, where ${\rm Id}[t]={\rm Id}\T\bC[t]$ commutes with $\msl_n[t]$
inside $\mgl_n[[t]]$.

In order to give the definition of the global Weyl module for $\mgl_n$ we need one more piece of notation. We define
$\varphi_k:{\rm Id}[t]\to \bC[x_1,\dots,x_k]$, ${\rm Id}\T t^i\mapsto x_1^i+\dots +x_k^i$.
Then the $\mgl_n[t]$ module $\bW_\la$ is defined as
\begin{equation}\label{Wmgln}
\mathbb{W}_{\lambda}=\mathbb{W}_{\bar\lambda} \otimes \U({\rm Id}[t])/\ker \varphi_{\la_n},
\end{equation}
where $\msl_n[t]$ acts on the first tensor factor and ${\rm Id}[t]$ on the second.

%\begin{rem}
%Yet another way to define $\bW_\la$ is as follows. Let us embed $\mgl_n$ into $\msl_{n+1}$ by the formula
%$A\mapsto A- E_{n+1,n+1}{\rm tr} A$, where $A\in\mgl_n$ and $E_{n+1,n+1}\in \mgl_{n+1}$ is the matrix unit located on the
%intersection of the last column and last row (in the lower right corner). Let $\mu=\sum_{i=1}^{n-1} (\la_i-\la_{i+1})\om_i + \la_n\om_n$
%be the $\msl_{n+1}$-weight. Then the $\mgl_n[t]$ Weyl module $\bW_\la$ can be realized inside $\msl_{n+1}[t]$ Weyl module $\bW_\mu$
%as the $\U(\mgl_n[t]$ closure of the cyclic vector $w_\mu\in \bW_\mu$.  We note that
%\[
%\A_\mu\simeq \A_{\bar\la}\otimes \U({\rm Id}[t])/\ker \varphi_{\la_n}.
%\]
%We also note that the algebra $\A_\mu$ acts freely on $\bW_\la$.
%\end{rem}
We define the right $\gl_n[t]$ Weyl modules in the same way as we did in the $\msl_n[t]$ situation.

\subsection{Current groups}
Our references in this section are \cite{Kum1,Kum2}.
Let $G$ be a simple algebraic group and let $\eO=\bC[[t]]$ be the ring of formal power series in a variable $t$.
Let $G(\eO)=G[[t]]$ be the corresponding group over $\eO$, sometimes referred to as the current group.
The group $G(\eO)$ is the set of algebra homomorphisms $\bC[G]\to\eO$. The current group can be
realized explicitly as follows. Let us consider a faithful representation of $G$ providing an embedding
$G\subset {\rm Mat}_n$ to the space of square matrices of size $n\times n$. Let $I_G$ be the defining ideal for this
embedding; in particular, $I_G$ can be realized as an ideal in the algebra of functions $R_n=\bC[{\rm Mat}_n]$.
The algebra $R_n$ can be naturally identified with the  polynomial ring $\bC[z_{i,j}]_{i,j=1}^n$
(where $z_{i,j}$ is a function on ${\rm Mat}_n$ returning the $(i,j)$-th entry of a matrix).

Let us introduce new infinite set of variables $z_{i,j}^{(k)}$, where $i,j=1,\dots n$ and $k\ge 0$.
We attach to each variable $z_{i,j}$ the formal power series $z_{i,j}(t)=\sum_{k\ge 0} z_{i,j}^{(k)}t^k$.
For a polynomial $P(z_{i,j})\in\bC[z_{i,j}]_{i,j=1}^n$ we denote by $P_m$ the coefficient of $t^m$ in $P(z_{i,j}(t))$, i.e.
\[
P(z_{i,j}(t))=\sum_{m\ge 0} P_m(z_{i,j}^{(k)})t^m.
\]
\begin{example}
Let $G=SL_2$ with the standard embedding into ${\rm Mat}_2$. Then $I_G$ is generated by a single polynomial
$z_{1,1}z_{2,2}-z_{1,2}z_{2,1}-1$. The polynomials $P_m$ are the coefficients of the series
$z_{1,1}(t)z_{2,2}(t)-z_{1,2}(t)z_{2,1}(t)-1$. For example,
\begin{gather*}
P_0=z_{1,1}^{(0)}z_{2,2}^{(0)}-z_{1,2}^{(0)}z_{2,1}^{(0)}-1,\\
P_1=z_{1,1}^{(0)}z_{2,2}^{(1)}+z_{1,1}^{(1)}z_{2,2}^{(0)}-z_{1,2}^{(0)}z_{2,1}^{(1)}-z_{1,2}^{(1)}z_{2,1}^{(0)},\\
P_2=z_{1,1}^{(0)}z_{2,2}^{(2)}+z_{1,1}^{(1)}z_{2,2}^{(1)}+z_{1,1}^{(2)}z_{2,2}^{(1)}-z_{1,2}^{(0)}z_{2,1}^{(2)}-z_{1,2}^{(1)}z_{2,1}^{(1)}-
z_{1,2}^{(2)}z_{2,1}^{(1)}.
\end{gather*}
\end{example}
We note that the polynomial ring $\bC[z_{i,j}^{(k)}]$, $i,j=1,\dots n$, $k\ge 0$ can be naturally identified with the algebra
of functions on ${\rm Mat}_n(\eO)={\rm Mat}_n[[t]]$.
Let $I_G(\eO)\subset \bC[{\rm Mat}_n(\eO)]$ be the ideal
generated by all polynomials $P_m$, $P\in I_G$, $m\ge 0$.
Then the $\bC$ points of the affine scheme associated to the quotient ring $\bC[z_{i,j}^{(k)}]/I_G(\eO)$
form the group $G(\eO)$. One gets an obvious embedding $G(\eO)\subset {\rm Mat}_n[[t]]={\rm Mat}_n\T\bC[[t]]$.
We note that the scheme structure defined by the ideal $I_G(\eO)$ is reduced (see \cite{Oo}).

\begin{rem}
Let $\bW_\la^*$ be the restricted dual module (the direct sum of duals of the homogeneous components with respect
to the $q$-grading). The group $G(\eO)$ acts on $\bW_\la^*$. However, it does not act on the Weyl module itself
(the action produces infinite sums). To fix this problem one has to consider the completed global Weyl module with respect
to the $q$-grading.  However, given a vector $v\in \bW_\la$, a functional $\xi\in\bW_\la^*$ and an element $A\in G(\eO)$
the matrix element $\xi(gv)$ is well defined.
\end{rem}

\section{Peter-Weyl theorem for current groups}\label{PWC}
\subsection{Characters}
Let $G$ be a simple algebraic group and let $G(\eO)=G[[t]]$ be the corresponding current group.
Let ${\rm ev}_0:G(\eO)\to G$ be the $t=0$ evaluation morphism. We also have a natural embedding $G\to G[[t]]$ as
the set of constant currents.
We denote by $G(\eO)_{{\rm id}}$  the preimage ${\rm ev}_0^{-1}$ of the identity element ${\rm id}\in G$.

\begin{lem}
$G(\eO)$  is isomorphic to the semi-direct product $G(\eO)_{{\rm id}}\rtimes G$. In particular, every element
$A\in G(\eO)$ can be uniquely written as $A=Bg$, where $g={\rm ev}_0 A\in G$, $B\in G(\eO)_{{\rm id}}$.
\end{lem}
\begin{proof}
Obviously, the group  $G(\eO)_{{\rm id}}$ is normal and for any $A\in G(\eO)$ one has $A ({\rm ev}_0 A)^{-1}\in G(\eO)$.
\end{proof}

The group $G(\eO)$ naturally acts on $\bC[G(\eO)]$ from the left and from the right, namely
\[
\left((g_1,g_2) \Psi\right)(A) = \Psi(g_1^{-1}Ag_2^{-1}),\ g_1,g_2,A\in G(\eO).
\]
In particular, if $A=Bg$, $B\in G(\eO)_{{\rm id}}$, $g\in G$, and $g_1,g_2\in G$, then
\[
\left((g_1,g_2) \Psi\right)(A) = \Psi\left((g_1^{-1}Bg_1)(g_1^{-1}gg_2^{-1})\right).
\]
Thus the $G\times G$ action on $G[[t]]$ is written as the standard
$G\times G$ action on $G$ and the diagonal action on $G(\eO)_{{\rm id}}$. We also note that there exists a natural $\bC^*$ action
on the space of functions $\bC[G(\eO)]$ given by the loop rotation (i.e. an element $z\in \bC^*$ scales $t$: $t\mapsto tz$).
We thus obtain the action of $G\times G\times \bC^*$ on $\bC[G(\eO)]$.
This gives $\bC[G(\eO)]$ the structure of graded $G \de G$ bimodue.
In what follows for a graded $G \de G$ bimodue
$U$ we use the notation $\ch_q U=\sum_{k\in\bZ} q^k \ch_{G\times G} U_k$, where $U_k\subset U$ consist of vectors $u$ such
that $\bC^*$ acts  on $u$ via the character $z\mapsto z^k$. Similarly, for the graded $G$ module $U$
we use the notation  $\ch_q U=\sum_{k\in\bZ} q^k \ch_{G} U_k$.

For a Lie algebra $\fa$ we denote by $t\fa[t]$ the Lie algebra $\fa\T t\bC[t]$.
Also for a vector space $V$ we write $V[t]$ for $V\T\bC[t]$ and $tV[t]$ for $V\T t\bC[t]$.
\begin{lem}\label{unip}
The character of the space of $\bC[G(\eO)_{{\rm id}}]$ with respect to the group $G\times \bC^*$ (with the diagonal
$G$ action) is equal to $\ch_q S(t\fg[t])$.
\end{lem}
\begin{proof}
The group $G(\eO)_{{\rm id}}$ is pro-unipotent (the projective limit of unipotent groups $G(N)_{{\rm id}}=G(\bC[t]/t^N)_{{\rm id}}$.
One has $\bC[G(N)_{{\rm id}}]=S(\fg\T\bC[t]/t^N)$. Now passing to the inductive limit when $N$ goes to infinity
we arrive at the claim of our Lemma.
\end{proof}

\begin{rem}\label{SU}
For any Lie algebra $\fa$ the symmetric algebra $S(\fa)$ is isomorphic to the universal enveloping algebra $\U(\fa)$ as $\fa$ modules
(with respect to the adjoint action). Hence, $\bC[G(\eO)_{{\rm id}}]$ is isomorphic to $\U(t\fg[t])$ as $G\times \bC^*$ modules.
\end{rem}

Let $(q)_\infty=\prod_{i\ge 1} (1-q^i)$.
\begin{cor}
The character of $\bC[G(\eO)_{{\rm id}}]$ is given by the formula
\[
(q)_\infty^{-{\rm rk}\fg}\prod_{\al\in \Delta_+} \prod_{k>0} (1-q^ke^\al)^{-1}(1-q^ke^{-\al})^{-1}.
\]
\end{cor}

Recall that $P_\la$ denotes the projective cover of the irreducible highest weight module $V_\la$.

\begin{prop}\label{func}
One has the isomorphism of graded $G \de G$ bimodues:
\[
\bC[G(\eO)]\simeq \bigoplus_{\la\in P_+} P_\la\T V_\la^o.
\]
\end{prop}
\begin{proof}
By  Lemma \ref{unip}, Remark \ref{SU} and the Peter-Weyl theorem one has the isomorphisms of graded $G \de G$ bimodues:
\begin{equation*}
\bC[G(\eO)]\simeq \bC[G(\eO)_{{\rm id}}]\T \bC[G]
\simeq \U(t\fg[t])\T\bigoplus_{\la\in P_+} V_\la\T V_\la^o
\simeq\bigoplus_{\la\in P_+} P_\la\T V^o_\la,
\end{equation*}
where the last isomorphism is implied by the isomorphism of $G\times \bC^*$-modules $P_\la\simeq \U(\fg[t])\T V_\la$.
\end{proof}

\begin{rem}
We note that in Proposition \ref{func} we consider only the $G \de G$ bimodue (not the whole current group action).
\end{rem}

We consider the local and global Weyl modules $W_\la$ and ${\mathbb W}_\la$.
\begin{lem}\label{PW}
One has the isomorphism of graded $G \de G$ bimodues
\[
\bigoplus_{\la\in P_+} P_\la\T V_\la^o=\bigoplus_{\la\in P_+} \bW_\la\T W_\la^o.
\]
\end{lem}
\begin{proof}
One has  the BGG-type reciprocity \cite{BBCKL,CI,Kh} for any $\la,\mu\in P_+$:
\[
[W_\la:V_\mu]_q=[P_\mu:\bW_\la]_q=[W^o_\la:V^o_\mu]_q,
\]
where $[\cdot,\cdot]_q$ denotes the $q$-multiplicity.
Therefore
\begin{align*}
\bigoplus_{\la\in P_+}  P_\la\T V_\la^o & \simeq \sum_{\la,\mu} [P_\la:\bW_\mu]_q\bW_\mu\T V_\la^o\\
& \simeq  \sum_{\la,\mu} [W^o_\mu:V^o_\la]_q \bW_\mu\T V_\la^o\\
& \simeq \bigoplus_{\la\in P_+} \bW_\la\T W^o_\la.
\end{align*}
\end{proof}

\subsection{Tensor products over highest weight algebras}
Recall the algebra $\A_\la$, which sits inside $\bW_\la$ as the algebra of highest weights (i.e.
$\A_\la$ is isomorphic to the weight $\la$ subspace of the global Weyl module $\bW_\la$).
The action of $\A_\la$ on $\bW_\la$ is free and the quotient
$\bW_\la/\A^+_\la$ is isomorphic to the local Weyl module $W_\la$.

\begin{lem}
The algebra $\A_\la$ acts freely on $\bW_\la^o$. The action is free and the quotient is isomorphic to $W_\la^o$.
\end{lem}
\begin{proof}
Let $\A_\la'$ be the weight $\la$ subspace of $W_\la^o$. This space has a structure of algebra via the surjective map
$w_\la^o.\U(t\fh[t])\to \A'_\la$. The defining relations of $\bW_\la^o$ imply the surjective homomorphism of algebras
$\A_\la\to\A'_\la$.
Now recall the isomorphism of vector spaces $\bW_\la^o\simeq \bW_{\la^*}$ (see Lemma  \ref{Wopp}).
Then $w_\la^o$ corresponds to the lowest weight
vector in $V_{\la^*}$ embedded into $\bW_{\la^*}$ as a $q$-degree zero subspace. Hence the characters of $\A_\la$ and $\A'_\la$
coincide.
\end{proof}

Now we introduce the main ingredient to formulate the $G(\eO)$ analogue of the Peter-Weyl theorem.
\begin{dfn}
We define the $\fg[t]\de \fg[t]$-bimodule
\[
T_\la=\bW_\la\T_{\A_\la} \bW_\la^o.
\]
\end{dfn}

Below we use the following simple observation:
\begin{lem}\label{Tcharacter}
We have the isomorphism of $q$-graded $\fg\de\fg$-bimodules:
\begin{equation}\label{TWW}
T_{\lambda}\simeq \bW_\la\T W_{\lambda}^o.
\end{equation}
\end{lem}

\begin{rem}
The isomorphism \eqref{TWW} does not hold as the isomorphism of $\fg[t]\de\fg[t]$-bimodules.
\end{rem}

\begin{lem}\label{relationsT}
$T_{\lambda}$ is a cyclic $\fg[t]\de\fg[t]$-bimodule with the generator $\bar w_{\lambda}=w_\la\T w_\la^o$ and the following relations:
\begin{gather*}
\fn_+[t]\bar w_\la=0,\quad  \bar w_\la\fn_-[t]=0,\\
ht^0\bar w_\la=\bar w_\la ht^0=\la(h)\bar w_\la,\ h\in\fh,\\
f_{-\alpha}^{\langle \la, \al^{\vee} \rangle+1}\bar w_{\lambda}=0,\
\bar w_{\lambda}e_{\alpha}^{\langle \la, \al^{\vee} \rangle+1}=0, \alpha \in \Delta_+,\\
h\T t^k \bar w_{\lambda} = \bar w_\la h\T t^k,\ h\in\fh, k\ge 0.
\end{gather*}
% and in $\gl_n[[t]]-\gl_m[[t]]$ case one more relation for action of the currents over $tr$.
\end{lem}
\begin{proof}
Follows from the definition of the module $T_\la$.
\end{proof}

\begin{prop}
One has the equality of the characters of graded $G \de G$ bimodues:
\[
\ch_q \bC[G(\eO)]=\bigoplus_{\la\in P_+}\ch_q \bW_\la\T_{\A_\la} \bW^o_\la.
\]
\end{prop}
\begin{proof}
This is a direct consequence of Lemma \ref{PW}, Proposition \ref{func} and Lemma \ref{Tcharacter}.
\end{proof}

\subsection{The Peter-Weyl theorem for current groups}
Let $\bC[G(\eO)]^*$ be the restricted dual space of functions. More precisely, let
$\bC[G(\eO)]=\bigoplus_{k\ge 0} \bC[G(\eO)]_k$ be the direct sum decomposition with respect to the loop rotation.
In particular, each space $\bC[G(\eO)]_k$ is a
$\fg$ module with respect to the left action and each irreducible $\fg$-module $V_\la$ shows up finite number of times.
Then we set
\[
\bC[G(\eO)]^*=\bigoplus_{k\ge 0} \bigoplus_{\la\in P_+} V_\la^*\T [\bC[G(\eO)]_k:V_\la].
\]
Our goal is to prove the following theorem:
%\begin{thm}

{\it There exists a filtration $F_\la$ on the dual space $\bC[G(\eO)]^*$ such that}
\[
{\rm gr} F_\bullet\simeq \bigoplus_{\la\in P_+} T_\la.
\]
%\end{thm}

Let us consider the standard order on $P_+$ defined by $\la\ge\mu$ if and only if $\la-\mu=\sum_{i=1}^r k_i\al_i$, $k_i\in \bZ_{\ge 0}$.
We construct a decreasing filtration $F_\la$ on the dual space of functions $\bC[G(\eO)]^*$ labeled by $\la\in P_+$.
Namely, let us consider the subspace $\bC[G]^*$.
By the Peter-Weyl theorem we have the direct sum decomposition $\bC[G]^*=\bigoplus_{\la\in P_+} V_\la\T V_\la^o.$
\begin{rem}
One has an obvious isomorphism of $\fg\de\fg$ bimodules $\bC[G]\simeq \bC[G]^*$.
However, this becomes wrong after passing from $G$ to $G[[t]]$. The reason we consider the dual space here
is explained below.
\end{rem}

We note that the embedding $G\subset G(\eO)$ induces the surjective restriction homomorphism
$\bC[G(\eO)]\to \bC[G]$ and hence the embedding $\bC[G]^*\subset \bC[G(\eO)]^*$.
We define
\[
\bC[G]^*_{\ge\la}=\bigoplus_{\mu\ge\la} V_\mu\T V_\mu^o
\]
and consider the right hand side as a subspace of $\bC[G(\eO)]^*$ via the embedding $\bC[G]^*\subset \bC[G(\eO)]^*$.

\begin{dfn}
Define the decreasing filtration $F_\la$, $\la\in P_+$ on $\bC[G(\eO)]^*$ by the formula
\[
F_\la= \U(\fn_-[t]) \bC[G]^*_{\ge\la} \U(\fb[t]),
\]
where the right and left actions of the current algebras are used.
\end{dfn}

Our goal is to prove the isomorphism of $\fg[t]\de \fg[t]$-bimoduules
\[
F_\la/\sum_{\mu > \la} F_\mu\simeq T_\la.
\]

We prepare several lemmas.
Let $v_\la\in V_\la$ be the weight $\la$ (highest weight) vector with respect to the left $\fg$ action and let $v_\la^o\in V_\la^o$
be the weight $\la$ vector with respect to the right action.
\begin{rem}
Recall the identification of vector spaces $V_\la^o\simeq V_\la^*$. By definition, weight $\la$ subspace of $V_\la^o$
corresponds to the weight $-\lambda$ subspace of $V_\la^*$, i.e. to the lowest weight subspace. Hence
$v_\la^o$ is the lowest weight vector in $V_\la^*$. In particular,
\begin{equation}\label{n+n}
\U(\fn_-)v_\la=V_\la,\ v_\la^o\U(\fn)=V_\la^o.
\end{equation}
\end{rem}

\begin{lem}
One has
\[
F_\la= \U(\fb_-[t]) (v_\la\T v_\la^o) \U(\fb[t])=\U(\fb_-[t]) (v_\la\T v_\la^o) \U(\fn[t]).
\]
\end{lem}
\begin{proof}
We note that for an element $h\in\fh$ and any $k\ge 0$ one has $(ht^k v_\la)\T v_\la^o=v_\la\T (v_\la^o ht^k)$.
In fact, this is equivalent  to $(X v_\la)\T v_\la^o=v_\la\T (v_\la^o X)$ for $X\in H[[t]]$, which holds since $v_\la$ and $v_\la^o$
have the same $\fh$-weight with respect to the left and right actions.  Now \eqref{n+n} completes the proof.
\end{proof}

The next lemma shows that the whole space $\bC[G(\eO)]^*$ is generated from the zero level subspace by the action of
$\fn_-\T t\bC[t]\oplus \fb_+\T t\bC[t]$.

\begin{prop}
One has
\[
\bC[G(\eO)]^*=\sum_{\la\in P_+} \U(t\fn_-[t])(V_\la\T V_\la^o) \U(t\fb_+[t]).
\]
\end{prop}
\begin{proof}
Assume that the right hand side is strictly contained in $\bC[G(\eO)]^*$.
Then there exists a function $\Psi\in\bC[G(\eO)]$ such that
\begin{equation}\label{Psi}
\Psi\left(\sum_{\la\in P_+} \U(t\fn_-[t])(V_\la\T V_\la^o) \U(t\fb_+[t])\right)=0.
\end{equation}
We note that the loop rotation invariants in $\bC[G(\eO)]$ coincide with the direct sum $\bigoplus_{\la\in P_+} V_\la\T V_\la^o$.
Therefore, we can (and will) assume that $\Psi$ is of strictly positive loop rotation degree. We will also
assume that this degree is the smallest possible (i.e. there is no function with the property \eqref{Psi}
of the loop rotation degree smaller than that of $\Psi$).
Equation \eqref{Psi} implies that
$\Psi$ is invariant with respect to the left-right action of the product of groups
$\exp(t\fn_-[t])\times \exp(t\fb_+[t])$. In other words, for any $g_t\in G(\eO)$ and $(A,B)\in  \exp(t\fn_-[t])\times \exp(t\fb_+[t])$
one has $\Psi(g_t)=\Psi(Ag_tB)$. In fact, \eqref{Psi} can be rewritten as
\[
(\U(t\fn_-[t])\Psi \U(t\fb_+[t]))\sum_{\la\in P_+} V_\la\T V_\la^o=0.
\]
Since $\Psi $ is of positive $q$-degree and of the smallest $q$-degree with the above property, we conclude
$\U(t\fn_-[t])\Psi \U(t\fb_+[t])=0$, which is equivalent to the claim that $\Psi$ is invariant with respect to the product
of groups $\exp(t\fn_-[t])\times \exp(t\fb_+[t])$.

We note that the set
\[
\exp(t\fn_-[t])G\exp(t\fb_+[t])\subset G(\eO)
\]
is open dense (recall that $G$ is considered as a subgroup of $G(\eO)$, the image of the $t=0$ evaluation map).
In fact, assume that $g_0={\rm ev}_0 g_t$ is an element from the open dense Bruhat cell $Bw_0B$ in $G$.
In particular,
\begin{equation}\label{oBc}
{\rm Ad}g_0. \fn_-\oplus \fb_+=\fg.
\end{equation}
Then the equality
$g_t=Ag_0B$ is equivalent to $g_t=g_0(g_0^{-1}Ag_0)B$. However, \eqref{oBc} implies
\[
{\rm Ad}g_0. t\fn_-[t]\oplus t\fb_+[t]=t\fg[t].
\]
Hence the kernel of ${\rm ev}_0$ can be written in the product from:
\[
G(\eO)_{{\rm id}}=(g_0^{-1} \exp(t\fn_-[t])g_0) \exp(t\fb_+[t]).
\]
Since $g_0^{-1}g_t\in G(\eO)_{{\rm id}}$, the presentation $g_t=g_0(g_0^{-1}Ag_0)B$ is possible.

Let $G(\eO)^0\subset G(\eO)$ be the open dense subset consisting of $g_t$ such that $g_0={\rm ev}_0g_t$ is in the open Bruhat cell.
Since the function $\Psi$ is invariant with respect to $\exp(t\fn_-[t])\times \exp(t\fb_+[t])$, we obtain for $g_t\in G(\eO)^0$:
\[
\Psi(g_t)=\Psi(Ag_0B)=\left((A^{-1}\times B^{-1})\Psi\right) (g_0)=\Psi(g_0).
\]
We conclude that $\Psi$ is invariant with respect to the $t=0$ evaluation morphism and hence its
q-degree (loop rotation degree) equals zero, which contradicts the assumption.
\end{proof}

Similar (even simpler) arguments imply the following modification of the Proposition above.
\begin{lem}
One has $\bC[G(\eO)]^*=\U(\fg[t]) \bC[G]^*$, i.e. the left action of the universal enveloping algebra of the
whole current algebra $\fg[t]$ on the degree zero subspace $\bC[G]^*\subset \bC[G(\eO)]^*$ produces the
whole dual space of functions on $G(\eO)$.
\end{lem}

Finally, we need the following lemma.
\begin{lem}\label{surjection}
The quotient space ${\rm gr} F_\la=F_\la/\sum_{\mu > \la} F_\mu$ carries a natural structure of $\fg[t]\de \fg[t]$ cyclic bimodule
with cyclic vector $v_\la\T v_\la^o$. The bi-module ${\rm gr} F_\la$ is a quotient of $T_\la$.
\end{lem}
\begin{proof}
We first note that $(\U(\fn_+[t])v_\la)\T v_\la^o$ belongs to $\sum_{\mu\ge \la} F_\mu$, since all the weights of $F_\la$ with respect
to the left action are smaller than or equal to $\la$. Similarly, $v_\la\T (v_\la^o \U(\fn_-[t]))$ sits inside $\sum_{\mu\ge \la} F_\mu$.
Since $F_\la$ is closed with respect to $\fb_-[t]\oplus\fb_+[t]$ action, we
conclude that $F_\la/\sum_{\mu > \la} F_\mu$ is a $\fg[t]\de\fg[t]$ bimodule.
The cyclicity is implied by the definition of $F_\la$. The last thing to check is that the defining relations of $T_\la$
from Lemma \ref{relationsT} are satisfied.
In fact, all the relations except for $a. v_\la\T v_\la^o = v_\la\T v_\la^o.a$ for any $a\in\A_\la$
are implied by the defining relations of the global Weyl modules. Let us prove the remaining relation.
Let us identify $v_\la\T v_\la^o\in V_\la\T V_\la^o\subset \bC[G]^*$ with the corresponding vector in $\bC[G(\eO)]^*$.
We need to show that for any $X\in H[[t]]$ and $\Psi\in\bC[G(\eO)]$ one has
\[
(v_\la\T v_\la^o) (X.\Psi)=(v_\la\T v_\la^o) (\Psi .X)
\]
for the left and right actions of the element $X$ on the function $\Psi$. This is equivalent to
$(v_\la\T v_\la^o) (\Phi)=(v_\la\T v_\la^o) (X^{-1}.\Phi .X)$ for $\Phi=X.\Psi$. Both left and right
hand sides depend only on the $q$-degree zero part $\Phi_0$ of $\Phi$. And for $\Phi_0$ the needed equality is clear since
the weights of $v_\la$ and $v_\la^o$ coincide.
\end{proof}

We conclude that the following theorem holds.
\begin{thm}\label{mainth}
The filtration $F_\la$ produces the desired isomorphism
\[
{\rm gr} \bC[G(\eO)]^*\simeq \bigoplus_{\la\in P_+} \bW_\la\T_{\A_\la} \bW_\la^o.
\]
\end{thm}
\begin{proof}
Lemma \ref{surjection} provides the surjection from the left hand side to the right hand side.
Now the theorem is implied by the character identities from Proposition \ref{func} and Lemma \ref{PW}.
\end{proof}

\section{Howe duality for current groups}\label{TypeA}
%\subsection{Characters}
For a partition $\lambda=(\la_1\ge\dots\ge \la_n\ge 0$) let  $p_\la(x_1,\dots,x_n;q)$ be the corresponding
q-Whittaker function. In particular, $p_\la$ are polynomials in $x_i$ and $q$ and can be defined as the $t=0$ specialization of the
symmetric Macdonald polynomial $P_\la$. The q-Whittaker functions turned out to be very important in modern representation theory
(see e.g. \cite{E,C,BF1,BF2}). In particular, one has the following lemma.

\begin{lem}
The character of the $\mgl_n[t]$ local Weyl module $W_\la$ is given by $p_\la(x,q)$.
\end{lem}
\begin{proof}
See \cite{I,S,GLO1}.
%Thanks to \cite{I,S} one knows that the characters of the local Weyl modules coincide with the $t=0$
%specialization of nonsymmetric Macdonald polynomials. However, the $t=0$ specializations of nonsymmetric and symmertic
%Macdonald polynomials do coincide the special(see ..............................).
\end{proof}

Assume that $m\ge n$. Then one has
the following  Cauchy type identity (we assume that $\la_{n+1}=0$):
\begin{equation}\label{CI}
\prod_{i=1}^n\prod_{j=1}^m(x_iy_j;q)_\infty^{-1}=
\sum_{\la_1\ge\dots\ge \la_n\ge 0} p_\la(x_1,\dots,x_n;q)p_\la(y_1,\dots,y_m;q) \prod_{i=1}^n (q)^{-1}_{\la_i-\la_{i+1}},
\end{equation}
where $(a;q)_\infty=\prod_{k\ge 0} (1-aq^k)$ and $(q)_k=\prod_{i=1}^k (1-q^k)$.
The identity \eqref{CI} is obtained by specializing at $t=0$ the corresponding identity for symmetric Macdonald polynomials
(see \cite{M}).
Now the left hand side of \eqref{CI} is equal to the character of the polynomials in variables $z_{i,j}^{(k)}$,
$i=1,\dots,m$, $j=1,\dots,n$ and
$k=0,1,\dots$.% (this is the space of functions on $GL_n[[t]]$, see  ......).

%\subsection{Modules}
We consider the left
action of the Lie algebra $\mgl_m[t]$  and the right action of the Lie algebra $\mgl_n[t]$ on the space of $m\times n$ matrices ${\rm Mat}_{m,n}(\bK[t])$.
Let $V$ be the left fundamental representation of the highest weight $\omega_1$ over $\mgl_m[t]$,
and let $U$ be the right fundamental representation of the highest weight $\omega_1$ over $\mgl_n[t]$.
We identify ${\rm Mat}_{m,n}(\bC[t])=V \otimes U \otimes \bC[t]$.

To a partition $\la$ with at most $r$ rows we attach a $\mgl_r$ weight $\sum \la_i\varepsilon_i$.
Then the simple roots are given by $\alpha_i=\varepsilon_i-\varepsilon_{i+1}$.
For two partitions $\la,\mu$ of size $N$ we write $\la\ge\mu$ if $\la-\mu$ is a sum of simple roots with nonnegative
integer coefficients.

Recall that $n \leq m$.
\begin{thm}\label{HoweCurrent}
The bimodule $S^N\left(V \otimes U \otimes \bC[t]\right)$ admits a decreasing filtration $F_{\lambda}$
indexed by partitions $\la$ of length $n$ with $\lambda \vdash N$ such
that
$$F_{\lambda}/F_{>\lambda}\simeq \bW_\la\T_{\A_\la} \bW^o_{\la}.$$
Moreover $S^N\left(V \otimes U \otimes \bC[t]\right)$ is generated by $S^N\left(V \otimes U \right)$
{\it as left $\gl_m[t]$-module}.
\end{thm}
\begin{proof}
We first note that the $q$-character of $S\left(V \otimes U \otimes \bC[t]\right)$ is given by the left hand side of
\eqref{CI} and the right hand side of \eqref{CI} is equal to the sum over all $\lambda$ of the characters of the spaces
$\bW_\la\T_{\A_\la} \bW^o_{\la^*}.$ In order to define $F_\lambda$ let us consider the classical Howe decomposition
\[
S^N(V\T U)=\bigoplus_{{\la\vdash N}} V_\la\T U_\la^o.
\]
Let $\bar v_\la=v_\la\T u_\la^o$ be the tensor product of weight $\la$ vectors. We define
\[
F_\la=\sum_{\mu\ge \lambda}\U(\mgl_m[t])\bar v_{\mu}\U(\mgl_n[t]).
\]
Then the first claim of Theorem follows from Lemma \ref{zerolevelgeneratedbimodule} below
(since $\fn_-[t]\bar v_{\lambda}\subset F_{>\lambda}$, $\bar v_{\lambda}\fn_+[t]\subset F_{>\lambda}$).
The second claim is proven in Lemma  \ref{zerolevelgeneratedleftmodule} below.
\end{proof}

We need the following simple Lemma.
Let $B=\mathrm{span}\{a_1, a_2\}$ be a two-dimensional $\msl_2$-module.
\begin{lem}\label{sl2Weyl}
The $\fn_-^{\msl_2}[t]$-module $S^N(B[t])$ is generated by $\mathrm{span}\left(\prod_{k=1}^N a_1 t^{r_k}, r_k\ge 0\right)$.
\end{lem}
\begin{proof}
We note that $\fn_-^{\msl_2}$ is spanned by a single element $f$ and $S^N(B[t])$ is the global Weyl module $\mathbb{W}_{N \omega}$.
The module $\bW_{N\omega}$ is generated
as $\fb_-^{\msl_2}[t]$-module by its highest weight element $a_1^N$.
Using the PBW theorem we obtain $\U(\fh[t])a_1^N=\mathrm{span}\left(\prod_{k=1}^N a_1 t^{r_k}, r_k\ge 0\right)$ and
\[
\U(\fn_-^{\msl_2}[t])\mathrm{span}\left(\prod_{k=1}^N a_1 t^{r_k}, r_k\ge 0\right)=S^N(B[t]).
\]
\end{proof}

Consider the $\gl_m\de \gl_n$ subbimodule $S^N(V \T U)$ inside $S^N(V \T U\T\bC[t])$.
In what follows we use the notation $\fb_+$, $\fb_-$ to denote the Borel
subalgebras in $\mgl_m$ or $\mgl_n$ (this does not lead to  confusion since $\mgl_m$ acts from the left and $\mgl_n$ acts
from the right).
\begin{lem}\label{zerolevelgeneratedbimodule}
$S^N(V \T U\T \bK[t])$ is generated by $S^N(V \T U)$ as $\fb_-[t]\de \fb_+[t]$ bimodule.
\end{lem}
\begin{proof}
Let $\{v_1, \dots, v_m\}$ be the weight basis of the left module $V$ such that $e_{ij}v_j=v_i$,
$\{u_1, \dots, u_n\}$ be the weight basis of the right module $U$ such that $u_ie_{ij}=u_j$.
We prove that the set $\{\prod_{i=1}^n(v_iu_i)^{r_{ii}}, r_{ii}\ge 0, \sum_i r_{ii}=N\}$ generates this bimodule.
Denote the linear span of these elements by $D$.

We divide the proof into three steps.

{\it Step 1.} Consider first the left action of elements $h_{ii}\otimes t^k$.
We have:
\[
\U(\fh[t])D=\mathrm{span} \left(\prod_{i=1}^n\prod_{k=1}^{r_{ii}}v_iu_it^{s_{ii}^{k}},\ s_{ii}^{k}\ge 0\right).\]
The proof of this fact is by induction on the number of nonzero exponents $s_{ii}^{k}$. Indeed assume by induction that $\U(\fh[t])D$
contains all elements of the form $x=\prod_{i=1}^n\prod_{k=1}^{r_{ii}}(v_iu_it^{s_{ii}^{k}})$ with the number of nonzero exponents $s_{jj}^{k}$
equal to $b<r_{jj}$. Then applying elements $h_{jj}t^a$ to $x$ we obtain the sum of elements such that all of them but one have
the number of nonzero exponents $s_{jj}^{k}$ equal to $b$. The remaining summand is equal to the integer multiple of
\[
\prod_{\substack{1\le i\le n\\i \neq j}}\prod_{k=1}^{r_{ii}}v_iu_it^{s_{ii}^{k}}\cdot
\prod_{k=1}^{r_{jj-1}}(v_ju_jt^{s_{jj}^{k}})\cdot (v_ju_jt^{{a}}).
\]
This completes the induction step.

{\it Step 2.} Next we prove the following claim:
\[\U(\fn_-[t])\U(\fh[t])D=
\mathrm{span}\left(\prod_{i\geq j}\prod_{k=1}^{r_{ij}}(v_iu_jt^{s_{ij}^{k}}) , s_{ij}^{k} \in \mathbb{Z}_{\geq 0}\right).\]

Assume by induction that all the elements $\prod_{i\geq j}\prod_{k=1}^{r_{ij}}(v_iu_jt^{s_{ij}^{k}})$ belong to
$\U(\fn_-[t])\U(\fh[t])D$
if $r_{ij}=0$ for $i \neq j$, $j<j_0$ or $j=j_0, i> i_0-1$ or $i<j$. We denote the linear span of these elements by $D_{i_0-1, j_0}$.

We prove that $D_{i_0, j_0}\subset\U(\fn_-[t])\U(\fh[t])D$
%if $r_{ij}=0$ for $j<j_0$ or $j=j_0, i> i_0$.

Consider the following filtration on the space $D_{i_0, j_0}$:
\[
G_p=\mathrm{span}\left(\prod_{i\geq j}\prod_{k=1}^{r_{ij}}(v_iu_jt^{s_{ij}^{k}}), r_{ij}=0 ~\text{for}~ j<j_0~ \text{or}~ j=j_0, i> i_0,
r_{i_0j_0}\leq p\right).
\]
Then this filtration preserves the structure of $\langle e_{i_0-1, i_0} \rangle[t]$-module.
Then the corresponding graded module is isomorphic to the direct sum of copies of modules $\mathbb{W}_{(r_{i_0j_0}+r_{i_0-1j_0})\omega}$.
Using Lemma \ref{sl2Weyl} we obtain that
the corresponding graded module is generated by $D_{i_0-1, j_0}$. This completes the induction step and Step 2 of the proof.

{\it Step 3.} Now we obtain the following:
\[
\U(\fn_-[t])U(\fh[t])D U(\fn_+[t])=S^N(V \T U\T \bC[t]).
\]
The proof is by analogous induction using the sets $D_{i_0j_0}$, $j_0>i_0$ which are the linear spans of the elements
$\prod_{i\geq j}\prod_{k=1}^{r_{ij}}(v_iu_jt^{s_{ij}^{k}})$, $r_{ij =0}$ for $i<i_0$, $j>j_0$.
Indeed we have:
\[D_{i_0,j_0}=D_{i_0j_0-1}U(\langle f_{i_0j_0} \rangle[t])\]
and by induction $D_{i_0,j_0} \subset \U(\fn_-[t])U(\fh[t])D U(\fn_+[t])$.
\end{proof}

\begin{lem}\label{zerolevelgeneratedleftmodule}
$S^N(V \T U\T \bC[t])$ is generated by $S^N(V \T U)$ as left $\gl_m[t]$ module.
\end{lem}

\begin{proof}
The proof is completely analogous. By Step 2 of the previous proof we have:
\[
\U(\fn_-[t])\U(\fh[t])D=
\mathrm{span}\left(\prod_{i\geq j}\prod_{k=1}^{r_{ij}}v_iu_jt^{s_{ij}^{k}}, s_{ik}^{k} \in \mathbb{Z}_{\geq 0}\right).
\]
Then the claim of the lemma is obtained by the same induction using the left action of $\fn_+[t]$.
\end{proof}

\section{Schur-Weyl duality}\label{SW}

\subsection{Classical Schur-Weyl duality}
Let us recall several different equivalent reformulations of the classical Schur-Weyl duality discovered by Issai Schur (\cite{Schur}) and popularized by Herman Weyl in \cite{HW_invariants}.
\begin{thm*}
There exists an isomorphism of $\GL_V\times \mathfrak{S}_n$-modules
\[V^{\otimes n} \simeq \oplus_{\substack{\lambda\vdash n\\l(\lambda)\le\dim V}} V_\lambda \otimes \mathbb{S}_\lambda.  \]
Here $V_\lambda$ denotes the irreducible $\GL_V$-module with highest weight $\lambda$ and $\mathbb{S}_{\lambda}$ is the irreducible $\mathfrak{S}_n$-module that corresponds to the partition $\lambda$ also known as the Specht module.
\end{thm*}
In particular, the Schur-Weyl duality defines a pair of functors that defines an equivalence between the category $\RepGL{V}{n}$ of polynomial $\gl_V$-representations of degree $n\leq m=\dim V$ and $\mathfrak{S}_n$-representations:
\begin{equation}
\label{eq::Schur-Weyl}
\begin{tikzcd}
\mathsf{Rep}(\mathfrak{S}_n) \arrow[rrr,shift left=1ex,"{M\mapsto M\otimes_{\mathfrak{S}_n} (V)^{\otimes n}}"]
 & & & \RepGL{m}{n} \arrow[lll," U_{(1^n,0^{m-n})}  \mapsfrom U "]
\end{tikzcd}.
\end{equation}
Here by $U_{(1^n,0^{m-n})}$ we denote the subspace of weight ${(\underbrace{1,\ldots,1}_n,\underbrace{0,\ldots,0}_{m-n})}$ in the $\GL_m$-module $U$. In the case $m>n$ the corresponding $\msl_m$ weight is equal to the $n$-th fundamental weight $\omega_{n}$.
In particular, the $\mathfrak{S}_n$-irreducible Specht module $\mathbb{S}_{\lambda}$ and irreducible $\GL_V$-module $V_\lambda$ associated with a Young diagram $\lambda$ are Schur-Weyl dual.

The following important Example~\ref{ex::Howe::SW} of Schur-Weyl dual modules was used by R.\,Howe in order to prove the equivalence of the Howe  and the Schur-Weyl dualities:
\begin{example}
\label{ex::Howe::SW}
Consider an auxiliary vector space $U$ and the $\GL_{U}$-equivariant version of Schur-Weyl duality~\eqref{eq::Schur-Weyl} between the category of $\GL_{U}\times\mathfrak{S}_n$ modules and $\GL_{U}\times\GL_V$ modules. The following modules becomes Schur-Weyl dual to each other:
\begin{equation}
\label{eq::Howe::SW}
\begin{tikzcd}
U^{\otimes n} = S^{n}(V\otimes U)_{(1^n,0^{m-n})}
 \arrow[r]
 & S^{n}(V\otimes U)  \simeq U^{\otimes n} \otimes_{\bC[\mathfrak{S}_n]} V^{\otimes n} \arrow[l].
\end{tikzcd}
\end{equation}
\end{example}

\subsection{Schur-Weyl duality for currents}
Let us proceed  with the representation categories over current algebras.
Namely, let us denote by $\RepGLt{m}{n}$ the category of finitely generated graded $\gl_m[t]$-modules (resp. $\bK[\mathfrak{S}_n]\ltimes \bK[t_1,\ldots,t_n])$-modules) with finite-dimensional graded components where, in addition, all graded components are polynomial $\gl_m$-representations of degree $n$. We denote by $\mathsf{Rep}(\mathfrak{S}_n\ltimes \bK[t_1,\ldots,t_n])$ the category of graded $\bK[\mathfrak{S}_n]\ltimes \bK[t_1,\ldots,t_n])$-modules.
\begin{lem}
The functors~\eqref{eq::Schur-Weyl} extend to the current case:
\begin{equation}
\label{eq::SW::currents}
\begin{tikzcd}
\mathsf{Rep}(\bK[\mathfrak{S}_n]\ltimes \bK[t_1,\ldots,t_n]) \arrow[rrr,shift left=1ex,"{M\mapsto M\otimes_{\bK[\mathfrak{S}_n]\ltimes \bK[t_1,\ldots,t_n]} (\bK^m[t])^{\otimes n}}"]
& & & \RepGLt{m}{n} \arrow[lll," U_{(1^n,0^{m-n})}  \mapsfrom U "]
\end{tikzcd}.
\end{equation}
\end{lem}
\begin{proof}
 The coincidence of the upper arrows in the classical case~\eqref{eq::Schur-Weyl} and in the current case~\eqref{eq::SW::currents} follows from the following isomorphism of
(graded) $\gl_m[t]$-modules:
\[
M\otimes_{\bC[\mathfrak{S}_n]\ltimes \bC[t_1,\ldots,t_n]} (\bC^m[t])^{\otimes n} \simeq M\otimes_{\mathfrak{S}_n} (\bC^m)^{\otimes n}.
\]
For the lower arrow let us first recall how one gets the action of $\mathfrak{S}_n$ on the ${(1^n,0^{m-n})}$-weight subspace in the classical case. Indeed, the normalizer $N(T_m)$ of the torus $T_m\subset \GL_m$ is isomorphic to $\mathfrak{S}_m\ltimes T_m$.
Therefore,  $\mathfrak{S}_n$ is considered as the quotient of the subgroup $\mathfrak{S}_n\ltimes T_n\subset N(T_m)$ that stabilizes
the given
$\msl_m\otimes 1$-weight  $\omega_{n}$ by the normal subgroup $T_n$. In the current case, there exists an analogous embedding of the algebra $\bK[\mathfrak{S}_n]\ltimes \bK[t_1,\ldots,t_n]$ into the subquotient
$\mathfrak{S}_n\ltimes T_n[[t]]/T_n$ and the latter inherits the natural action on the weight subspace  $U_{(1^n,0^{m-n})}$.
\end{proof}

\begin{thm}
\label{thm::SW::currents}	
Assuming $m\ge n$, the functors~\eqref{eq::SW::currents} define an equivalence of categories.
\end{thm}
\begin{proof}
The proof is separated into Propositions~\ref{prp::GL::HWC},\ref{prp::projectivegenerator},\ref{prp::automoprhismalgebra} below.
Let us explain the strategy:
First, in Proposition~\ref{prp::GL::HWC} we explain why the category $\RepGLt{m}{n}$ is a highest weight category in the sence recalled in Section~\ref{sec::HWC}. Second, the BGG reciprocity for this category is used in Proposition~\ref{prp::projectivegenerator}
where we prove that for any auxiliary vector space $U$ the $\gl_{V}[t]$-module $S^{n}(U\otimes V\otimes\bK[t])$ is a projective module
in the category  $\RepGLt{m}{n}$. Moreover, for $\dim U>n$ the direct summand $V[t]^{\otimes n}\subset S^{n}(U\otimes V\otimes\bK[t])$
is a projective generator of the category $\RepGLt{m}{n}$.
Finally, in Proposition~\ref{prp::automoprhismalgebra} we show that the algebra of endomorphisms of the latter projective generator is isomorphic to the algebra $\bK[\mathfrak{S}_n]\ltimes \bK[t_1,\ldots,t_n]$ and
the classical Morita Theorem (\cite{Mo}) implies the equivalence of categories~\eqref{eq::SW::currents}.
\end{proof}	

Let us notice several important properties of the right-hand side category of~\eqref{eq::SW::currents} that are crucial for the proof.
Denote by $\RepSLt{m}{n}$ the category of $\mathfrak{sl}_m[t]$ representations with graded components being polynomial representations
of degree $n$. Recall that $\dim V=m$.
\begin{prop}
\label{prp::GL::HWC}	
Assume that $\dim V\geq n$ and fix a number $r>n$. Then the category $\RepGLt{V}{n}$ is isomorphic to the Serre subcategory
$\Rep(\msl_r[t])^{\leq n\omega_1}\subset \Rep(\msl_r[t])$ generated by simple finite-dimensional $\msl_r$-modules whose highest weight
is less than or equal to $n\omega_1$.	
\end{prop}
\begin{proof}
Irreducible polynomial $\GL_V$-representations of degree $n$ are indexed by Young diagrams $\lambda\vdash n$, with $l(\lambda)\leq \dim V$
and the corresponding highest weight is equal to $\sum_{i=1}^{m} \lambda_i \varepsilon_i$. Since $m\ge n$ the condition on the
length $l(\lambda)\leq \dim V$ can be omitted.
Similarly, for $r>n$ the same set $\{\sum_{i=1}^r \lambda_i \varepsilon_i | \lambda_1\ge \ldots\ge \lambda_m\}$ indexes the set of dominant integral $\msl_r$-weights that are less or equal then $n\omega_1\simeq n\varepsilon_1$ with respect to the standard partial ordering
$\leq_{\msl}$ (compare with~\eqref{eq::def::ordering} defined for a general semisimple Lie algebra):
\[
\lambda\leq_{\msl} \mu \Leftrightarrow (\mu - \lambda) \text{ is the } \mathbb{Z}_{\geq 0}\text{-sum of positive simple roots }
\varepsilon_i - \varepsilon_{i+1}.
\]
In particular, the weight $\omega_{n}\simeq \varepsilon_1+\ldots+\varepsilon_n$ is the minimal element in this set.
Therefore, for $r>n$ and $\dim V\geq n$ we have an equivalence of categories $\Rep(\GL_V)^{(n)}$ and $\Rep(\msl_r)^{\leq n\omega_1}$
which implies the desired equivalence of corresponding categories of representations of current algebras.
\end{proof}

Thanks to Corollary~\ref{cor::HW::leq} and Example~\ref{ex::SL::HW} we know that
the category $\Rep(\msl_m[t])^{\leq n\omega_1}$ is a highest weight category, whose simples are irreducible $\msl_m$-modules $V_{\lambda}$ indexed by partitions $\lambda\vdash n$, standard modules coincide with global Weyl modules $\bW_{\lambda}$, proper standard are
local Weyl modules and projectives $P_{\lambda}^{\leq n\omega_1}$ are the quotients of $P_{\lambda}$ by the submodules
generated by subspaces of the "wrong" $\msl_m$-weights $\mu\not\leq n\omega_1$. In particular, the BGG reciprocity
\eqref{cor::BGG} is valid. Namely, projective cover  $P_{\lambda}^{\leq n\omega_1}$ admits a filtration with successive
quotients isomorphic to standard modules $\bW_{\mu}$  with  $\lambda \leq \mu \leq n\omega_1$, and the multiplicities satisfy
the following identity:
\begin{equation}
\label{eq::BGG:ne1}
{ [P_{\lambda}^{\leq n\omega_1}: \bW_{\mu}]= [W_{\mu}:V_{\lambda}] }  =[P_{\lambda}: \bW_{\mu}].
\end{equation}
From now on we suppose $\dim V=m>n$ in order to be able to deal with a highest weight category $\Rep(\msl_m[t])^{\leq n\omega_1}$.
\begin{prop}\label{prp::projectivegenerator} Assume $\dim V>n$ , $\dim U \geq n$.
	The $\msl_{V}[t]$-module $S^n(V \otimes U \otimes \bK[t])$ is a projective module in the category $\Rep(\msl_V[t])^{\leq n \omega_1}$.
\end{prop}	
\begin{proof}
Thanks to the classical Howe duality
\[
S^N(V\otimes U) = \bigoplus_{\substack{\la\vdash N\\	l(\lambda) \leq \min(\dim(V),\dim(U))}} V_{\lambda} \otimes U_{\lambda}
\]
we know that $\msl_V\otimes 1$-irreducible submodules of the module under consideration belong to $\Rep(\msl_V)^{\leq n \omega_1}$.
The restriction $l(\lambda) \leq \min(\dim(V),\dim(U))$ can be omitted thanks to the assumption $\dim V,\dim U\geq n$.
Theorem~\ref{HoweCurrent} implies that, in addition, the $\msl_V[t]$-module $S^n(V \otimes U \otimes \bK[t])$ is generated by the component of zero degree $\oplus_{\lambda\vdash n} V_\lambda\otimes U_{\lambda}$.
Hence, there exist a surjection of $\msl_V[t]$-modules from a projective cover of the zero level:
\begin{equation}
\label{eq::p::suv}
\pi: \oplus_{\lambda\vdash n} (P_{\lambda}^{\leq n\omega_1})^{\oplus \dim U_{\lambda}} \twoheadrightarrow S^n(V \otimes U \otimes \bK[t]).
\end{equation}
Moreover, the $q$-graded $\msl_V$-character of this module was also computed in~\eqref{HoweCurrent} as a part of the Howe duality for currents:
\[
\ch_{q} S^N(V \otimes U \otimes \bC[t]) = \sum_{\lambda\vdash N} \ch_q\bW_{\lambda}(V) \dim_q W_{\lambda}(U). \]
Here by $\dim_q$ we denote the graded dimension of a vector space.
The BGG reciprocity~\eqref{eq::BGG:ne1} for the highest weight category $\Rep(\msl_V[t])$ and its subcategory
$\Rep(\msl_V[t])^{\leq n\omega_1}$ implies the following simplification of characters:
\begin{multline*}
\ch_{q} S^n(V \otimes U \otimes \bK[t]) \stackrel{\text{\ref{HoweCurrent}}}{=}
\sum_{\lambda\vdash n}\ch_q(\bW_{\lambda}(V)) \dim_q W_{\lambda}(U) =\\
=\sum_{\lambda\vdash n}\ch_q(\bW_{\lambda}(V)) \left( \sum_{\mu\leq \lambda} \dim_q U_{\lambda} \cdot  [W_{\lambda}(U):U_{\mu}]\right) \stackrel{\text{\eqref{eq::BGG:ne1}}}{=}
\\ =\sum_{\mu\leq\lambda\leq n\omega_1} [P_{\mu}:\bW_{\lambda}]\cdot \ch_q \bW_{\lambda} \dim_q U_{\mu}  \stackrel{\text{\eqref{eq::BGG:ne1}}}{=} \sum_{\mu\leq n\omega 1} \ch_q P_{\mu}^{\leq n\omega_1} \dim_q U_{\mu}.
\end{multline*}
Consequently, the surjection $\pi$ in~\eqref{eq::p::suv} is an isomorphism.
\end{proof}

\begin{cor}
\label{cor::proj_gen}
Let $\dim V>n$.
Then $\msl_V[t]$-module $V[t]^{\otimes n}$ is a projective generator in the category $\Rep(\msl_V[t])^{\leq n\omega_1}$.	
\end{cor}	
\begin{proof}
Take $r=\dim U$ such that $\dim V \geq \dim U > n$.
Notice that the  $\msl_V[t]$-module $S^n(V\otimes U\otimes\bC[t])$ admits a direct sum decomposition with respect to the
$\msl_U$-weight subspaces. In particular, each summand is a projective module in the subcategory $\Rep(\msl_V[t])^{\leq n\omega_1}$.
As mentioned in~\eqref{eq::Howe::SW} (with $U$ and $V$ interchanged) the $\GL_U$ weight subspace of weight
$(1^n,0^{r-n})$ of the module $S^n(V\otimes U\otimes\bC[t])$ is isomorphic to the $\gl_V[t]$-module $V[t]^{\otimes n}$ implying the projectivity of the latter one.
Moreover, we have the following decomposition as $\msl_V[t]$-modules:
\begin{multline*}
V[t]^{\otimes n} \simeq (S^n(V\otimes U\otimes\bC[t]))_{(1^n,0^{r-n})} \simeq
%\\ \simeq
\oplus_{\lambda\vdash n} P_{\lambda}^{\leq n\omega_1} \otimes (U_{\lambda})_{(1^n,0^{r-n})} \simeq \oplus_{\lambda\vdash n}
P_{\lambda}^{\leq n\omega_1} \otimes \mathbb{S}_{\lambda}.
\end{multline*}
Here $U_{\lambda}$ denotes the irreducible $\gl_r$-module of highest weight $\lambda$ and by $\mathbb{S}_{\lambda}$ we
denote the Specht module which is different from zero for all $\lambda\vdash n$. In particular, the projective covers of all
irreducibles appears as summands of the projective module $V[t]^{\otimes n}$. Hence this module is a projective generator.
\end{proof}	

\begin{prop}
	\label{prp::automoprhismalgebra}
The map
\begin{equation}
\label{eq::morph::proj}
 \psi: (\prod_{i=1}^{n} t_i^{a_i})\cdot \sigma {\longrightarrow} \left(v_1 t^{b_1}\otimes \ldots v_n t^{b_n}
\mapsto v_{\sigma^{-1}(1)} t^{a_1+b_{\sigma^{-1}(1)}} \otimes \ldots \otimes v_{\sigma^{-1}(n)} t^{a_n+b_{\sigma^{-1}(n)}}\right)
\end{equation}	
defines an isomorphism of the degenerate affine Hecke algebra $\bK[\mathfrak{S}_n] \ltimes \bK[t_1, \dots, t_n]$ and the algebra of endomorphisms of the $\msl_V[t]$-module $V[t]^{\otimes n}$.
\end{prop}
\begin{proof}	
We omit the straightforward computations showing that, first, for all monomials $m$ the correspondence
$\psi(m)$ is an automorphism of the projective module $V[t]^{\otimes n}$ and, second, that $\psi$ is a map of algebras.
The nontrivial part explained below is to show that $\psi$ is an isomorphism.

Each morphism $\varphi$ from a projective cover $P_{\lambda}$ of an irreducible $\gl_V[t]$-module $V_{\lambda}$ to a given module $M$ is uniquely defined by the $\msl_V$-equivariant morphism $\varphi:V_{\lambda}\to M$. In particular, it is defined by the image of the highest weight vector $v_\lambda\in V_{\lambda}$.
Therefore, each endomorphism $\varphi\in  {\rm End}_{\gl_V[t]}(V[t]^{\otimes n})$ has a unique decomposition
$\sum_{\lambda\vdash n} \varphi_{\lambda}$, where $\varphi_{\lambda}(v_\mu) =0$ whenever $\mu\neq\lambda$ and $\varphi_\lambda$ is
defined by the image of the highest weight
vectors of $V^{\otimes n}$ of weight $\lambda$.
Consider the decomposition
\[{\rm End}_{\GL(V)}(V^{\otimes n}) \simeq \bC[\mathfrak{S_n}] = \oplus_{\lambda\vdash n} {\rm Mat}_{\dim \mathbb{S}_{\lambda}}(\bK).\]
Let $\{v_{T}\}$ be the basis of the space of highest weight vectors in $V^{\otimes n}$ of weight $\lambda\vdash n$ which is
known to be isomorphic to the Specht module $\mathbb{S}_{\lambda}$ and is indexed by standard Young tableaux of shape $\lambda$.
Thanks to the $\msl_V$-equivariance we know that the element $\varphi_{\lambda}(v_T)$ can be presented as the sum
\[\sum_{T'\vdash \lambda} v_{T'} f_{T'}(t_1,\ldots,t_n)\]
for appropriate polynomials $f_T'(t_1,\ldots,t_n)\in\bK[t_1,\ldots,t_n]$ and hence, $\varphi_{\lambda}$ uniquely defines an element
$\bar{\varphi}_{\lambda}\in {\rm Mat}_{\dim \mathbb{S}_{\lambda}}(\bC[t_1,\ldots,t_n])$.
The isomorphism of graded vector spaces
\[
\bC[\mathfrak{S_n}]\ltimes \bK[t_1,\ldots,t_n] = \oplus_{\lambda\vdash n} {\rm Mat}_{\dim \mathbb{S}_{\lambda}}(\bC[t_1,\ldots,t_n])
\]
finishes the proof of the proposition.
\end{proof}

\subsection{Applications}
Now we are able to reprove and specify certain results from \cite{Kat2} and \cite{BKM}:
\begin{cor}
	The category of graded representations of the degenerate affine Hecke algebra $\bK[\mathfrak{S_n}]\ltimes \bK[t_1,\ldots,t_n] $ is a highest weight category in the sence of Section~\S\ref{sec::HWC}.
\end{cor}	
We denote the corresponding standard modules $\Delta(\lambda)$ by \emph{global Kato modules} $\mathbb{K}_{\lambda}$ and the corresponding proper standard modules $\bar{\Delta}(\lambda)$ by \emph{local Kato modules} $K_{\lambda}$. The Schur-Weyl duality~\eqref{eq::SW::currents} maps the global/local Weyl module $\bW_\lambda$/$W_\lambda$ to the global/local Kato module $\mathbb{K}_{\lambda}$/$K_\lambda$ and we end up with the following corollary of the equivalence of categories.
\begin{cor}
$\phantom{1}$	
\begin{enumerate}
	\item
	\label{item::A_lambda}
	The algebra of endomorphisms of the global Kato module $\mathbb{K}_{\lambda}$ is isomorphic to $\A_{\lambda}$.
	\item\label{item::free}
	 $\mathbb{K}_{\lambda}$ is a free $\A_{\lambda}$-module with the graded set of generators isomorphic to $K_\lambda$.
	\item
	\label{item::Ch::K}
	The graded $\mathfrak{S}_n$-character of $K_\lambda$ is equal to the $q$-Whittaker function $p_{\lambda}(x,q)$.
\end{enumerate}
\end{cor}
\begin{proof}
Note that our equivalence of categories preserves the graded characters of Schur-Weyl dual modules.
Now, items~\eqref{item::A_lambda}, \eqref{item::free} follow from the fact that the algebras of endomorphism of a module and its
Morita dual module are isomorphic in Morita equivalent categories.
	Item~\eqref{item::Ch::K} follows from the fact that symmetric functions assigned to Schur-Weyl dual representations are the same and consequently the equivalence of categories we consider preserves the graded characters of Schur-Weyl dual modules.
\end{proof}	
The standard substitution $q=1$ in a $q$-Whittaker function $p_{\lambda}(x;q)|_{q=1} = e_{\lambda}$
implies the following
\begin{cor}
\begin{itemize}	
\item One has an isomorphism of $\mathfrak{S}_n$-modules  $K_\lambda$ and the induced module
${\rm Ind}_{\mathfrak{S}_{\lambda^{t}}}^{\mathfrak{S}_n} Sgn$. Here the subgroup
$\mathfrak{S}_{\lambda^{t}}=\mathfrak{S}_{\lambda_1^{t}}\times\mathfrak{S}_{\lambda_2^{t}} \times \ldots$ is the subgroup permuting the columns of the Young tableaux of shape $\lambda$.
\item The dimension of $K_{\lambda}$ is equal to the multinomial coefficient $\binom{|\lambda|}{\lambda_1^{t}, \dots, \lambda_n^{t}}$.
\end{itemize}	
\end{cor}
%\begin{proof}	
%	The nongraded character ($q=1$) of the local Weyl module is equal to the character of the tensor product of evaluation modules and is equal to $e_\lambda = e_{\lambda_1^t} \ldots e_{\lambda_n^t}$ (see e.g~\cite{CL}). The latter is equal to the $\mathfrak{S}_n$-character of the aformentioned induced representation.
%\end{proof}	

Let us finish this Section with the following Theorem \ref{Schur-Weylcurrent} that can be considered as a natural current version of the Schur-Weyl duality:
\begin{thm}\label{Schur-Weylcurrent}
For each $n \in \mathbb{N}$ there exists a decreasing filtration $\{F_{\lambda}| \lambda\vdash n, l(\lambda)\le n\}$ of
$\gl_V[t]\de\bC[\mathfrak{S}_n] \ltimes \bC[t_1, \dots, z_t]$-bimodule
$V[t]^{\otimes n}$ such that:
\[
{\rm gr} V[t]^{\otimes n}\simeq \bigoplus_{\lambda\vdash n, l(\lambda)\le n}\mathbb{W}_{\lambda} \otimes_{\A_{\lambda}}\mathbb{K}_{\lambda}.
%	 \begin{cases}	\mathbb{W}_{\lambda} \otimes_{\A_{\lambda}}\mathbb{K}_{\lambda}, \text{ if } l(\lambda)\leq \dim V, \\ 0, \text{ otherwise. }\end{cases}
	\]
\end{thm}
\begin{proof}
Consider an auxiliary vector space $U$ with  $\dim U>n$.
Generalizing the classical equivariant Schur-Weyl duality between Howe module and tensor power mentioned in~\eqref{eq::Howe::SW} we conclude that the $\gl_V[t]$-equivariant Schur-Weyl duality defines the correspondence between
$\gl_V[t]\de\bK[\mathfrak{S}_n] \ltimes \bK[t_1, \dots, t_n]$ and $\gl_V[t]\de\gl_U[t]$-bimodules: 	
\[
V[t]^{\otimes n}\leftrightarrow S^n(V\otimes U\otimes\bC[t]).
\]
Theorem~\ref{HoweCurrent} explains the existence of the desired filtration for the  $\gl_V[t]\de\gl_U[t]$-bimodule with successive
subquotients isomorphic to $\bW_\la(V)\otimes_{\A_{\lambda}}\bW_{\la}^{o}(U)$. The equivalence of categories obtained in
Theorem~\ref{Schur-Weylcurrent} implies the existence of corresponding filtration for
$\gl_V[t]\de\bK[\mathfrak{S}_n] \ltimes \bK[t_1, \dots, t_n]$-bimodule where the global Weyl module $\bW_\lambda(U)$ is replaced
by the global Kato module $\mathbb{K}_{\lambda}$.
% Note that the Specht modules $\mathbb{S}_\lambda$ are selfdual and therefore we do need to take opposite for the Kato modules as it is in the case of Howe duality.
\end{proof}

\subsubsection{Combinatorics of local Kato modules}
\label{sec::Kato::fillings}
In this subsection we derive a combinatorial formula for the characters of Kato modules using combinatorics of Weyl modules
described in \cite{FeMa2}.
\begin{cor}
\label{cor::local::Kato::fillings}
	The basis of the local Kato module $K_{\lambda}$ is indexed by the coloumn-decreasing fillings (without repetitions) of a Young diagram $\lambda$ by numbers $1,\dots,n$. The $q$-degree of a filling is defined in~\cite{FeMa2} (Definition 2.6) and is recalled below in Definition~\ref{def::snake::rule}.
\end{cor}
\begin{proof}
 The only thing one has to do is to  get a description of the $(1^{n},0^{m-n})$-weight subspace of a local Weyl module $W_{\lambda}$ whose graded basis was constructed in~\cite{FeMa2}.
\end{proof}
\begin{definition}
\label{def::snake::rule}	
The degree of a column-decreasing filling $f$ of a Young diagram $\lambda\vdash n$ by numbers $1,\dots,n$ is the following sum
\[
\sum_{1\leq i<j \leq l(\lambda^t)} k(f(\lambda_i^t),f(\lambda_j^t)),
\]
where $f(\lambda_i^t)$ is the filling of the $i$-th column and the number $k(\sigma,\tau)$ of two filled columns  $\sigma=(\sigma_1> \dots> \sigma_l)$, $\tau=(\tau_1>\dots>\tau_r)$ computes the number of sign changes in the row inequalities written in the reversed order:
\[
\sigma_l \bigvee \tau_l; \ \sigma_{l-1} \bigvee \tau_{l-1};\ldots ; \sigma_1\bigvee \tau_1.
\]
If the length of $\sigma$ is greater than the length of $\tau$ then we assume that $\sigma_l < \emptyset $.
Moreover, if $\sigma_l>\tau_l$, then we say that the sign does change at the $l$-th position.
\end{definition}
Let us provide an example of computations.
\begin{example}
	\label{snakeexample}
The numbers $k(\sigma,\tau)$ are written under the pairs of columns $\sigma$ and $\tau$ and the sign changing inequalities are circled:
\[
\begin{array}{cccccc}
&{	\begin{tikzpicture}[scale=0.7]
	\node (v00) at (-.5,2.5) { {{1}}};
	\node (v10) at  (-.5,1.5) {{2}};
	\node (v20) at (-.5,0.5) {{8}};
	\node (v30) at (-.5,-.5) {{10}};
	\node (v01) at (0.5,2.5) {$<$};
	\node (v11) at  (.5,1.5) {{$<$}};
	\node (v21) at (.5,0.5) {$<$};
	\node (v31) at (.5,-.5) {$<$};
	\draw[step=1cm] (-1,-1) grid (0,3);
	\node (v02) at (1.5,2.5) {{4}};
	\node (v12) at  (1.5,1.5) {{5}};
	\node (v22) at (1.5,0.5) {{9}};
	\node (v32) at (1.5,-.5) {};
	\draw[step=1cm] (1,0) grid (2,3);
	\end{tikzpicture}
} &	\quad;\quad &
{	\begin{tikzpicture}[scale=0.7]
	\node (v00) at (-.5,2.5) { {{1}}};
\node (v10) at  (-.5,1.5) {{2}};
\node (v20) at (-.5,0.5) {{8}};
\node (v30) at (-.5,-.5) {{10}};
	\node (v01) at (0.5,2.5) {$<$};
	\node (v11) at  (.5,1.5) {\textcircled{$<$}};
	\node (v21) at (.5,0.5) {\textcircled{$>$}};
	\node (v31) at (.5,-.5) {$<$};
	\draw[step=1cm] (-1,-1) grid (0,3);	
	\node (v02) at (1.5,2.5) {{3}};
	\node (v12) at  (1.5,1.5) {{6}};
	\node (v22) at (1.5,0.5) {{7}};
	\node (v32) at (1.5,-.5) {};
	\draw[step=1cm] (1,0) grid (2,3);
	\end{tikzpicture}
} &	\quad;\quad &
{		\begin{tikzpicture}[scale=0.7]
		\node (v00) at (-.5,2.5) { {{4}}};
		\node (v10) at  (-.5,1.5) {{5}};
		\node (v20) at (-.5,0.5) {{9}};
%		\node (v30) at (-.5,-.5) {\bf{5}};
		\node (v01) at (0.5,2.5) {\textcircled{$>$}};
		\node (v11) at  (.5,1.5) {\textcircled{$<$}};
		\node (v21) at (.5,0.5) {\textcircled{$>$}};
%		\node (v31) at (.5,-.5) {$<$};
		\draw[step=1cm] (-1,0) grid (0,3);
		\node (v02) at (1.5,2.5) {{3}};
		\node (v12) at  (1.5,1.5) {{6}};
		\node (v22) at (1.5,0.5) {{7}};
%		\node (v32) at (1.5,-.5) {};
		\draw[step=1cm] (1,0) grid (2,3);
		\end{tikzpicture}
}\\
k(\sigma,\tau) =  & 0 & & 2 & & 3
\end{array}
\]
We also provide an example of a computation of the $q$-degree of the following filling of the Young diagram $\lambda=(3,3,3,1)= (4,3,3)^t$:
\[
\begin{tikzpicture}[scale=0.5]
\draw[-] (-.3,3-1) edge[bend right=20] (-.3,-1-1);
\node (k) at (-2,1-1) {\large $\deg_q$};
	\draw[step=1cm] (0,-1) grid (3,2);
	\draw[step=1cm] (0,-2) grid (1,-1);
\node (v00) at (.5,2.5-1) { {{1}}};
\node (v10) at  (.5,1.5-1) {{2}};
\node (v20) at (.5,0.5-1) {{8}};
\node (v30) at (.5,-.5-1) {{10}};
\node (v01) at (1.5,2.5-1) { {{4}}};
\node (v11) at  (1.5,1.5-1) {{5}};
\node (v21) at (1.5,0.5-1) {{9}};
\node (v02) at (2.5,2.5-1) {{3}};
\node (v12) at  (2.5,1.5-1) {{6}};
\node (v22) at (2.5,0.5-1) {{7}};
\draw[-] (3.3,3-1) edge[bend left=20] (3.3,-1-1);
\node (p) at (4.5,0) {\Large {$=$}};
\end{tikzpicture}
\begin{tikzpicture}[scale=0.5]
\draw[step=1cm] (0,-1) grid (1,3);
\node (v00) at (.5,2.5) { {{1}}};
\node (v10) at  (.5,1.5) {{2}};
\node (v20) at (.5,0.5) {{8}};
\node (v30) at (.5,-.5) {{10}};
\draw[-] (-.3,3) edge[bend right=20] (-.3,-1);
\node (k) at (-1.2,1) {\large $k$};
\node (p) at (1.2,1) {\large{,}};
\end{tikzpicture}
\begin{tikzpicture}[scale=0.5]
\draw[step=1cm] (0,0) grid (1,3);
\node (v01) at (.5,2.5) { {{4}}};
\node (v11) at  (.5,1.5) {{5}};
\node (v21) at (.5,0.5) {{9}};
\node (v31) at (.5,-.75) {};
\draw[-] (1.3,3) edge[bend left=20] (1.3,-1);
\node (p) at (2.5,1) {\large {$+$}};
\end{tikzpicture}
\begin{tikzpicture}[scale=0.5]
\draw[-] (-.3,3) edge[bend right=20] (-.3,-1);
\node (k) at (-1.2,1) {\large $k$};
\draw[step=1cm] (0,-1) grid (1,3);
\node (v00) at (.5,2.5) { {{1}}};
\node (v10) at  (.5,1.5) {{2}};
\node (v20) at (.5,0.5) {{8}};
\node (v30) at (.5,-.5) {{10}};
\node (p) at (1.2,1) {\large{,}};
\end{tikzpicture}
\begin{tikzpicture}[scale=0.5]
\draw[step=1cm] (0,0) grid (1,3);
\node (v01) at (.5,2.5) { {{3}}};
\node (v11) at  (.5,1.5) {{6}};
\node (v21) at (.5,0.5) {{7}};
\node (v31) at (.5,-.75) {};
\draw[-] (1.3,3) edge[bend left=20] (1.3,-1);
\node (p) at (2.5,1) {\large {$+$}};
\end{tikzpicture}
\begin{tikzpicture}[scale=0.5]
\draw[-] (-.3,3) edge[bend right=20] (-.3,-1);
\node (k) at (-1.2,1) {\large $k$};
\draw[step=1cm] (0,0) grid (1,3);
\node (v01) at (.5,2.5) { {{4}}};
\node (v11) at  (.5,1.5) {{5}};
\node (v21) at (.5,0.5) {{9}};
\node (p) at (1.2,1) {\large{,}};
\end{tikzpicture}
\begin{tikzpicture}[scale=0.5]
\draw[step=1cm] (0,0) grid (1,3);
\node (v01) at (.5,2.5) { {{3}}};
\node (v11) at  (.5,1.5) {{6}};
\node (v21) at (.5,0.5) {{7}};
\draw[-] (1.3,3) edge[bend left=20] (1.3,-1);
\node (p) at (2.5,1) {\Large {$=$}};
\node (p0) at (5.5,1) {{$0+2+3 = 5$}};
\end{tikzpicture}
%= 0+2+3 = 5
\]
\end{example}
The combinatorial basis of the local Kato module described in Corollary~\ref{cor::local::Kato::fillings} can be used,
for example, in order to prove the following
\begin{lem}
\label{lem::K_la::top}	
\begin{equation}
\label{eq::deg:q:K_la}
\dim_q K_{\la} = q^{\sum_i \binom{\lambda_i}{2}}+ \text{ lower terms }. \end{equation}
\end{lem}
\begin{proof}
Let us find the upper bound on the $q$-degree of the Kato module. In terms of filled diagrams we have to find the upper bound
of the number of sign changes in the snake rule. Note that each pair of cells in one row may produce at most one sign change assigned to the snake rule for the corresponding columns. Therefore, the $q$-degree is bounded from above by the number of pairs of cells in one row,
which is equal  to the degree of the right-hand side of~\eqref{eq::deg:q:K_la}, denoted by $d(\lambda):=\sum_i \binom{\lambda_i}{2}$.
It remains to explain that there exists a unique filling of this degree $d(\lambda)$ of a diagram $\lambda$.

Suppose $T$ is a filling of a Young diagram $\lambda\vdash n$ by integers $\{1,\ldots,n\}$. Recall that the filling increases
downstairs in each column. Consider a pair of cells located on the intersection of the  $s$'th and $t$'th columns with
the $k$'th row ($s<t$). This pair of cells will affect the degree of the filling if and only if one has the inequality (sign) change for this pair of cells and the two corresponding cells below.
The inequality change may happen in one of the following possible cases:
\begin{equation}
\label{eq::cells::pairs}
\begin{array}{ccccc}
{
\begin{tikzpicture}[scale=0.85]
\node (v00) at (-.5,.5) {\tiny{$a_{ks}$}};
\node (v10) at  (-.5,-.5) {\tiny{$a_{k+1 s}$}};
\node (v01) at (0.5,.5) {$<$};
\node (v11) at (0.5,-.5) {$>$};
\draw[step=1cm] (-1,-1) grid (0,1);
\node (v02) at (1.5,.5) {\tiny{$a_{kt}$}};
\node (v12) at  (1.5,-.5) {\tiny{$a_{k+1 t}$}};
\draw[step=1cm] (1,-1) grid (2,1);
\end{tikzpicture}
} &  &
{
	\begin{tikzpicture}[scale=0.85]
	\node (v00) at (-.5,.5) {\tiny{$a_{ks}$}};
	\node (v10) at  (-.5,-.5) {\tiny{$a_{k+1 s}$}};
	\node (v01) at (0.5,.5) {$>$};
	\node (v11) at (0.5,-.5) {$<$};
	\draw[step=1cm] (-1,-1) grid (0,1);
	\node (v02) at (1.5,.5) {\tiny{$a_{kt}$}};
	\node (v12) at  (1.5,-.5) {\tiny{$a_{k+1 t}$}};
	\draw[step=1cm] (1,-1) grid (2,1);
	\end{tikzpicture}
} &  &
{
	\begin{tikzpicture}[scale=0.85]
	\node (v00) at (-.5,.5) {\tiny{$a_{ks}$}};
	\node (v10) at  (-.5,-.5) {\tiny{$a_{k+1 s}$}};
	\node (v01) at (0.5,.5) {$>$};
%	\node (v11) at (0.5,-.5) {$<$};
	\draw[step=1cm] (-1,-1) grid (0,1);
	\node (v02) at (1.5,.5) {\tiny{$a_{kt}$}};
%	\node (v12) at  (1.5,-.5) {\tiny{$a_{k+1 t}$}};
	\draw[step=1cm] (1,0) grid (2,1);
	\end{tikzpicture}
} \\
a_{ks}<a_{kt}<a_{k+1 t} <a_{k+1 t} & , &
a_{kt}<a_{ks}<a_{k+1 s} <a_{k+1 t} & , &
a_{kt}<a_{ks}<a_{k+1 s}.
\end{array}
\end{equation}
The rightmost case corresponds to the inequality $\lambda_{k+1}<t$ when there is no cell below $a_{kt}$.
In particular, one can see that all integers in the $k$'th row of $T$ are strictly less than integers that appear in the $k+1$'st row. Consequently, one has to fill the first row with integers $\{1,\ldots,\lambda_1\}$, integers
$\{\lambda_1+1,\ldots,\lambda_1+\lambda_2\}$ have to be placed in the second row, and so on. The filling in the last row has
to be the reverse lexicographical, and, moreover, the restriction~\eqref{eq::cells::pairs} uniquely defines the ordering of the filling of the $k$'th row whenever the filling of the $k+1$'st row is given. Thus, there exists a unique filling of the Young diagram of the top degree.
\end{proof}	

\begin{lem}
\label{lem::Sgn::Kato}
The multiplicity of the Sign $\mathfrak{S}_n$-representation in the local Kato module $K_\lambda$%, considered as $\mathfrak{S}_n$-representation
 is equal to $1$.
Moreover, the Sign representation appears in $K_\la$ in the top graded component.
\end{lem}	
\begin{proof}
The Frobenius reciprocity implies
\[\dim \Hom_{\mathfrak{S}_n}(K_{\la},Sgn_n) =
\dim \Hom_{\mathfrak{S}_n}(Ind_{\mathfrak{S}_{\la^t}}^{\mathfrak{S}_n} Sgn, Sgn_n) = \dim \Hom_{\mathfrak{S}_{\la^t}}(Sgn,Sgn) = 1.
\]	
Similarly, the multiplicity of the trivial representation in $K_{\la}$ is equal to $0$ if $\lambda\neq n\varepsilon_1$ and is equal
to $1$ if $\lambda$ is a partition with exactly one row. But the trivial representation corresponds to the maximal partition
with respect to the partial ordering. In the latter case the trivial representation corresponds to the generator of the Kato module
$K_{(n)}$ and $q$-degree is equal to $0$.
Therefore, there exists almost unique one-dimensional $\mathfrak{S}_n$-subrepresentation of $K_{\la}$ that appears in a nonzero $q$-degree.
 On the other hand Lemma~\ref{lem::K_la::top} explains that the top-dimensional component of $K_{\la}$ is one-dimensional, finalizing
the proof of Lemma~\ref{lem::Sgn::Kato}.
\end{proof}	

\begin{thm}
\label{thm::Lambda::SW}
The $\gl_V[t]$ module $\Lambda^n(V[t])$ admits a filtration by global Weyl modules with the following graded multiplicities
 $[\Lambda^n(V[t]):\bW_\la] = q^{\sum_j\binom{\lambda_j}{2}}.$
\end{thm}
\begin{proof}
The $\gl_V[t]$-module $\Lambda^n(V[t])\simeq \Hom_{\mathfrak{S}_n}\left( Sgn_n,V[t]^{\otimes n}\right)$ is a direct summand of
$V[t]^{\otimes n}$. Consequently, it inherits the filtration by global Weyl modules discovered in Theorem~\ref{Schur-Weylcurrent}. The multiplicity of the global Weyl module $\bW_\la$  in $\Lambda^n(V[t])$ is equal to the multiplicity of the sign representation in the local Kato module $K_\lambda$. The latter was computed in Lemma~\ref{lem::Sgn::Kato} using combinatorics of the basis of $K_\lambda$.
\end{proof}	
The graded characters of modules considered in Theorem~\ref{thm::Lambda::SW} gives the following equality, which is the
specialization of the identity VI.5.4 of~\cite{M}.
\begin{cor}
The following identity for $q$-Whittaker functions holds:
\[
\prod_{\substack{k \geq 0\\ i=1,\dots,n}}(1+x_iq^k)=
\sum_{\lambda}\frac{q^{d(\lambda)}}{(q)_{\lambda}}p_{\lambda}(x,q), \text{ where } d(\la) = \sum_j\frac{\lambda_j(\lambda_j-1)}{2}.
\]
\end{cor}

\section*{Acknowledgments}
We are grateful to Alexei Borodin, Alexander Braverman, Boris Feigin, Michael Finkelberg, Syu Kato and Shrawan Kumar
for useful discussions and correspondence.
The research of E.F. was supported by the grant RSF 19-11-00056.
A.Kh. research has been funded by the Russian Academic Excellence Project '5-100',
%Results of  Section~\S\ref{sec::Rational::MONR::all}, (in particular, Theorem~\ref{thm::HICG}) have been obtained under support of the RSF grant No.19-11-00275.
A.Kh. is a Young Russian Mathematics award winner and would like to thank its sponsors and jury.
The work of Ie.M. was supported by Japan Society for the Promotion of Science.

\end{document}